\documentclass[12pt]{amsart}
\usepackage[margin=2.0cm]{geometry}

\usepackage[backend=biber,maxnames=50,
  giveninits=true]{biblatex}
\AtEveryBibitem{%
	\clearfield{number}}

\usepackage{orcidlink}

\usepackage{amssymb}
\usepackage{amsfonts}
\usepackage{mathrsfs}
\usepackage{graphicx}
\usepackage{mathtools}

\usepackage{float}

\usepackage{hyperref}
\hypersetup{colorlinks,linkcolor={blue!75!black},citecolor={blue!75!black},urlcolor={blue!75!black}}

\newcommand{\R}{{\mathbb R}}
\newcommand{\N}{{\mathbb N}}

\newtheorem{theorem}{Theorem}[section]

\newtheorem{corollary}[theorem]{Corollary}
\newtheorem{lemma}[theorem]{Lemma}
\newtheorem{proposition}[theorem]{Proposition}
\newtheorem{conjecture}[theorem]{Conjecture}

\theoremstyle{definition}
\newtheorem{definition}[theorem]{Definition}

\newtheorem{question}[theorem]{Question}

\theoremstyle{remark}
\newtheorem{remark}[theorem]{Remark}

\DeclareMathOperator{\diam}{diam}
\DeclareMathOperator{\conv}{conv}
\DeclareMathOperator{\pos}{pos}

\DeclareMathOperator{\vol}{vol}
\DeclareMathOperator{\GL}{GL}
\DeclareMathOperator{\SL}{SL}

\let\emptyset\varnothing

\usepackage{enumitem}

\newcommand{\Knn}{\mathcal{K}_n^n}
\newcommand{\Ktt}{\mathcal{K}_2^2}
\newcommand{\transp}{{\mkern-1.5mu\mathsf{T}}}

\DeclarePairedDelimiter\scpr{\langle}{\rangle}
\newcommand{\sphere}{\mathbb{S}^{n-1}}
\newcommand{\sphereone}{\mathbb{S}^{1}}

\newcommand{\Sulambda}{S_u^\lambda}
\newcommand{\norm}[1]{\|#1\|}

\DeclareMathOperator{\linh}{span}
\newcommand{\Prob}{\mathbb{P}}
\newcommand{\dint}{\,\mathrm{d}}
\DeclareMathOperator{\aff}{aff}
\DeclareMathOperator{\interior}{int}
\DeclareMathOperator{\relint}{relint}
\DeclareMathOperator{\SO}{SO}
\DeclareMathOperator{\Id}{Id}

\newcommand{\cdotbox}{\hspace{0.4mm}\cdot\hspace{0.4mm}}

\newcommand{\MID}{\,\middle|\,}
\newcommand{\E}{\mathbb{E}} 
\newcommand{\dotnorm}{\|\cdot\|}
\newcommand{\dId}{d_{\{\Id\}}}
\newcommand{\deltaId}{\delta_{\{\Id\}}}
\DeclareMathOperator{\Orthogonal}{O}

\DeclareRobustCommand{\One}{\text{\usefont{U}{bbold}{m}{n}1}}

\DeclareMathOperator{\bd}{bd}
\newcommand{\goestoinfty}[1]{\underset{#1 \rightarrow \infty}{\longrightarrow}}

\title{Random convergence of generalized Steiner processes}

 \author[C. Kipp]{Christian Kipp$^{\orcidlink{0009-0001-6029-6902}}$}

\address{Technion -- Israel Institute of Technology, Department of Mathematics, Haifa 32000, Israel}
\email{ckipp@campus.technion.ac.il}

\author[\'A. Sagmeister]{\'Ad\'am Sagmeister$^{\orcidlink{0000-0002-6863-6481}}$}

\address{\parbox{\linewidth}{Bolyai Institute, University of Szeged, Aradi v\'ertan\'uk tere 1, 6720 Szeged, Hungary\\HUN-REN Alfr\'ed R\'enyi Institute of Mathematics, Re\'altanoda u. 13-15, 1053 Budapest, Hungary}}

\email{sagmeister@server.math.u-szeged.hu, sagmeister.adam@gmail.com}


\mathtoolsset{showonlyrefs=true}

\begin{document}

\begin{abstract}
    It is well-known that a random Steiner process converges almost surely to a Euclidean ball. On the other hand, as we show in this article, a planar random shaking process converges almost surely to a triangle (in the Banach--Mazur distance). Asking whether these phenomena are generic or exceptional, we study random iterations of an operation called $\lambda$-Steiner symmetrization, which includes Steiner symmetrization and shaking as special cases. We establish abstract criteria for almost sure convergence and divergence of such processes, obtaining criteria for many other symmetrization operations from convex geometry as a byproduct. Based on a joint characterization of ellipses and triangles, we give a full resolution to our question in the planar case.
\end{abstract}

\maketitle

\setcounter{tocdepth}{1}

\tableofcontents

\section{Introduction} \label{section:introduction}

For the purposes of this article, we define a \emph{convex body} to be a convex compact set with non-empty interior; the set of convex bodies in $\R^n$ is denoted by $\Knn$. Throughout the article, we do not really distinguish between a convex body and its translates. We equip $\Knn$ with the \emph{geometric distance}
    \begin{equation}
        d(K,L)\coloneqq \inf \{ \lambda\mu\mid \lambda,\mu\geq 1,\;\exists y,z \in \R^n\colon \lambda (K-y)\supset L-z \supset \mu^{-1} (K-y)\},
    \end{equation}
and when we refer to topological concepts like continuity, the underlying topology is generated by $\delta \coloneqq \log d$, which is a metric on $\Knn$ modulo translations. By a \emph{geometric functional}, we mean a translation-invariant, continuous map $\Knn \rightarrow \R$. %
In order to prove a geometric inequality for such a functional (and to identify the extremal bodies), one often applies a suitable \emph{operation}, i.e., a family of maps $\Knn \rightarrow \Knn$.

The classical example of an operation is \emph{Steiner symmetrization}, which was developed by Steiner to solve the classical isoperimetric problem. Let $K \in \Knn$ and let $\pi_H$ be the orthogonal projection onto a hyperplane $H$. For every direction $u \in \sphere$, there exists a unique pair of convex functions $f_{K,u},g_{K,u} \colon \pi_{u^\bot}(K) \rightarrow \R$ with
\begin{equation}
    K = \{x \in \R^n \mid \pi_{u^\bot}(x) \in \pi_{u^\bot}(K) \text{ and } f_{K,u}(\pi_{u^\bot}(x)) \leq \scpr{x,u} \leq -g_{K,u}(\pi_{u^\bot}(x))\}.
\end{equation}
Conversely, the convex body $K$ is uniquely determined by $\pi_{u^\bot}(K)$ and the pair $f_{K,u},g_{K,u}$. The \emph{Steiner symmetral of $K$ in direction $u$} can be defined as the unique convex body $S_uK$ with
\begin{equation} \label{eq_def_steiner}
    \pi_{u^\bot}(S_uK) = \pi_{u^\bot}(K) \quad \text{and} \quad f_{S_uK,u}=g_{S_uK,u}=\tfrac12 (f_{K,u} + g_{K,u}).
\end{equation}
Intuitively speaking, $S_uK$ is obtained from $K$ by moving all chords in direction $u$ in their respective affine hulls until they are precisely cut in half by the hyperplane $u^\bot$. The relevance of this operation stems from the fact that numerous geometric functionals $\varphi$ are \emph{monotone} with respect to Steiner symmetrization, i.e., they satisfy $\varphi(S_uK)\leq\varphi(K)$ for all $u$ and $K$, and that any convex body $K \in \Knn$ can be brought arbitrarily close to a Euclidean ball by applying a suitable sequence of Steiner symmetrizations. This provides a powerful generic proof strategy to establish that a geometric functional $\varphi$ is minimized by Euclidean balls, and this method has been used to prove numerous important isoperimetric-type inequalities; for examples, we refer to \cite[Ch.~9]{gar06}, \cite[Ch.~9]{gru07} and the references therein. Considering the remarkable properties of Steiner symmetrization, it is natural to ask to what extent these properties are coincidental or, alternatively, reflect deeper principles. A natural approach to this question is to compare Steiner symmetrization with other similar operations, as done by Bianchi--Gardner--Gronchi \cite{bgg17,bgg22} and Ulivelli \cite{uli23}. %

Motivated by Sylvester's four-point problem, a closely related operation was introduced by Blaschke in \cite{bla17}. The \emph{shaking of $K$ in direction $u$} can be defined as the unique convex body $B_uK$ with
\begin{equation} \label{eq_def_blaschke}
    \pi_{u^\bot}(B_uK) = \pi_{u^\bot}(K), \quad f_{B_uK,u}=0 \quad \text{and} \quad g_{B_uK,u}=f_{K,u} + g_{K,u}.
\end{equation}
The motivation to introduce shaking was Blaschke's observation that the expected volume of a random triangle in a planar convex body does not decrease under shaking, and that any planar convex body $K \in \Ktt$ can be brought arbitrarily close to a triangle by applying a suitable sequence of shakings. This resolves Sylvester's four-point problem in the plane, identifying triangles as the maximizers, by an argument that exhibits a remarkable duality to the technique of Steiner symmetrization.
In higher dimensions, it was shown by Campi--Colesanti--Gronchi that any $K \in \Knn$ can be brought arbitrarily close to any simplex with non-obtuse angles \cite{ccg01} by applying a sequence of shakings. The analogy with Steiner symmetrization is not complete: Since an $n$-dimensional simplex fails to be invariant under shaking for $n\geq 3$, any functional that is monotone under shaking cannot be uniquely maximized by simplices in dimension $n \geq 3$, and it seems  like such functionals are either trivial or rather exotic. On the other hand, shaking in well-chosen directions can be helpful. Overall, the method of shaking and discrete variants of it have been useful in a handful of further contexts \cite{sch73,kle79,bl91,uhr94,bb01,fl22}.

A goal of this article is to study Steiner symmetrization and shaking in a unified framework. 
In view of the similarity of the two definitions \eqref{eq_def_steiner} and \eqref{eq_def_blaschke}, it is natural to construe the two operations as the extreme points of a 1-parameter family of operations, which is introduced in the following definition.

\begin{definition} \label{definition:lambda_steiner_symmetrization}
    Let $K \subset \R^n$ be a convex body and let $\lambda \in [0,\frac12]$. We define the \emph{$\lambda$-Steiner symmetral of $K$ in direction $u$} as the unique convex body $\Sulambda K$ with
    \begin{equation} \label{eq_def_lambda_steiner}
    \pi_{u^\bot}(\Sulambda K) = \pi_{u^\bot}(K), \quad f_{\Sulambda K,u}= \lambda (f_{K,u} + g_{K,u}) \quad \text{and} \quad g_{\Sulambda K,u}=(1-\lambda)(f_{K,u} + g_{K,u}).
    \end{equation}
\end{definition}
Broadly speaking, our goal is to understand which of the  features of Steiner symmetrization and shaking are generic in the sense that they are satisfied by $\lambda$-Steiner symmetrization for generic values of $\lambda$. To make this question more concrete, we focus on the convergence behavior of $\lambda$-Steiner processes, i.e., sequences of convex bodies that result from iterated  applications of $\lambda$-Steiner symmetrization. If the directions $u_1,u_2,\dots$ are chosen independently from the rotation-invariant probability measure on $\sphere$, we call the resulting random sequence $(S^\lambda_{u_i}\cdots S^\lambda_{u_1} K)_{i \in \N}$ the \emph{random $\lambda$-Steiner process starting at $K$}. In the special cases $\lambda=\frac12$ and $\lambda=0$, we call the random sequence the \emph{random Steiner process} and the \emph{random shaking process starting at $K$}, respectively. Our starting point is the fact that the random Steiner process starting at $K$ converges to a Euclidean ball almost surely, which was shown by Mani-Levitska \cite{ml86} (see \cite{vs06crit,v13,cd14,bbgv12,bf13} for generalizations and variations of this result). In view of this fact, we ask the following, still somewhat broad question:

\begin{question} \label{question:broad_question}
    Let $K \in \Knn$ and $\lambda \in [0,\frac12]$. What can be said about convergence of the random $\lambda$-Steiner process starting at $K$?
\end{question}

Considering that any planar convex body can be brought arbitrarily close to a triangle by a sequence of shakings, and assuming that there is some common principle that connects Steiner symmetrization and shaking, %
one might conjecture that any planar random shaking process converges almost surely to a triangle. However, although shaking a triangle yields another triangle, the second triangle is generally not a translate of the first. This suggests to consider a notion of convergence that does not distinguish between the two triangles. Indeed, we will see later that the random shaking process converges almost surely to the triangle in the Banach--Mazur distance (this follows from Corollary~\ref{corollary:random_shaking_process_converges_to_any_triangle}). More generally, considering any closed subgroup $G$ of $\GL(n,\R)$, we define $K,L\in \Knn$ to be \emph{$G$-equivalent} and write $K \sim_G L$ if there exists an $A \in G$ such that $AK$ is a translate of $L$. We denote the equivalence class of $K$ by $[K]_G$ and say that a sequence of convex bodies $K_1,K_2,\dots \in \Knn$ \emph{converges modulo $G$} if it converges in the topological quotient space $(\Knn,\log d)/\mathord\sim_G$. In light of this terminology, we can rephrase Question \ref{question:broad_question} more precisely:

\begin{question} \label{question:precise_question}
    Let $K \in \Knn$, $\lambda \in [0,\frac12]$ and let $G \leq \GL(n,\R)$ be closed. What is the probability that the random $\lambda$-Steiner process starting at $K$ converges modulo $G$?
\end{question}

In the first part of this article, we seek to establish criteria for almost sure convergence and divergence of the random $\lambda$-Steiner process. Since the arguments we employ are mostly topological, it is convenient to study the problem in a somewhat more abstract setting. For our purposes, it is practical to introduce the notion of a \emph{Steiner operation}, by which we mean a family of translation-class preserving, inclusion-monotone and positively homogeneous maps $\Knn \rightarrow \Knn$, which is parameterized by $\Orthogonal(n)$ in a rotation-equivariant way (see Definition~\ref{definition:steiner_operation} for a precise definition). This notion encompasses $\lambda$-Steiner symmetrization as well as Schwarz symmetrization, Minkowski symmetrization and other well-known symmetrization operations in convex geometry (see Remark~\ref{remark:examples_of_steiner_operations}). Basic properties of these operations are discussed in Section~\ref{section:basic_properties_of_steiner_operations}.

If $\Sigma$ is a Steiner operation, we denote the \emph{forward orbit of $K \in \Knn$ under $\Sigma$} by
\begin{equation}
    \Sigma^*K=\{\Sigma_{U_k}\cdots\Sigma_{U_1}K\mid k \in \N, U_1, \dots,U_k \in\Orthogonal(n)\}
\end{equation}
and its topological closure by $\overline{\Sigma^*K}$. We refer to a sequence of the form $(\Sigma_{U_i}\cdots\Sigma_{U_1}K)_{i\in \N}$ as a \emph{generalized Steiner process}; if the matrices $U_i$ are chosen independently according to the Haar probability measure on $\Orthogonal(n)$, we call the resulting random sequence the \emph{random $\Sigma$-process starting at $K \in \Knn$} (see Definition~\ref{definition:random_sigma_process_starting_at_K}).

In Section~\ref{section:random_convergence_of_generalized_steiner_processes}, we state and prove criteria for almost sure convergence and divergence of generalized Steiner processes, obtaining conditions for $\lambda$-Steiner processes as a special case. The precise conditions require notions that we defer to Section~\ref{section:basic_properties_of_steiner_operations}, but let us shortly look at some deterministic special cases in order to illustrate how the conditions relate to the examples of random Steiner processes and random planar shaking processes. For this, we need two definitions.

The first definition generalizes the special roles that Euclidean balls and non-obtuse simplices play for Steiner symmetrization and shaking, respectively.

\begin{definition}
    Let $\Sigma$ be a Steiner operation. We say that $L^*\in \Knn$ is \emph{$\Sigma$-universal} if for every $K\in \Knn$ there exists a sequence $U_1,U_2,\dots \in \Orthogonal(n)$ and a number $r>0$ with
    \begin{equation}
        \delta(\Sigma_{U_i}\cdots\Sigma_{U_1}K,rL^*) \goestoinfty{i} 0.
    \end{equation}
\end{definition}

We have already mentioned that shaking a triangle yields another triangle, and it is obvious that the Steiner symmetral of a Euclidean ball $B$ is a translate of $B$. The second definition provides a suitable notion of invariance for general Steiner operations.

\begin{definition} \label{definition:sigma_invariant}
    Let $\Sigma$ be a Steiner operation and let $G\leq\GL(n,\R)$ be closed. We say that $L^* \in \Knn$ is \emph{$\Sigma$-invariant modulo $G$} if $\Sigma_{U_k}\cdots\Sigma_{U_1}L^*\in [L^*]_G$ for every finite sequence $U_1,\dots,U_k\in \Orthogonal(n)$. 
\end{definition}

With these definitions at hand, we can now state the deterministic special cases of our more general criteria. In both criteria, $G$ is an arbitrary closed subgroup of $\GL(n,\R)$. The first criterion concerns almost sure convergence of generalized Steiner processes.%
\begin{theorem} \label{theorem:sufficient_condition_for_convergence_simple_version}
    Let $\Sigma$ be a Steiner operation and let $K\in\Knn$ such that $\overline{\Sigma^*K}$ is compact. If there exists a $\Sigma$-universal body $L^*\in \Knn$ that is $\Sigma$-invariant modulo $G$, then the random $\Sigma$-process $\Sigma_{U_k}\cdots\Sigma_{U_1}K$, $k \in \N$, converges modulo $G$ almost surely and its limit is almost surely $G$-equivalent to a homothet of $L^*$.
\end{theorem}

In this article, Theorem~\ref{theorem:sufficient_condition_for_convergence_simple_version} and its more general version (Theorem~\ref{theorem:characterization_tethered}) serve as tools to solve Question~\ref{question:precise_question}, but the results can be detached from this purpose. From a broader perspective, Theorem~\ref{theorem:sufficient_condition_for_convergence_simple_version} shows a deeper principle behind Mani-Levitska's result \cite{ml86}, and extends it to many other symmetrization operations (see Corollary~\ref{corollary:convergence_to_ball_for_examples}).

The second criterion provides a deterministic sufficient condition for almost sure divergence of certain random $\Sigma$-processes.

\begin{theorem} \label{theorem:sufficient_condition_for_divergence_simple_version}
    Let $\Sigma$ be a Steiner operation and let $K\in\Knn$ such that $\overline{\Sigma^*K}$ is compact. If there exists no $\Sigma$-invariant body $L^*\in \Knn$ modulo $G$, then the random $\Sigma$-process $\Sigma_{U_k}\cdots\Sigma_{U_1}K$, $k \in \N$, diverges modulo $G$ almost surely.
\end{theorem}

We note that Theorem~\ref{theorem:sufficient_condition_for_convergence_simple_version} and Theorem~\ref{theorem:sufficient_condition_for_divergence_simple_version} cannot resolve Question~\ref{question:precise_question}, because the forward orbit of $K \in \Knn$ under shaking is never precompact for $n \geq 2$ (see Section \ref{section:compact_recurrence_of_planar_lambda_steiner_processes}). The more general criteria from Section~\ref{section:random_convergence_of_generalized_steiner_processes} encompass arbitrary planar $\lambda$-Steiner processes; instead of requiring $\overline{\Sigma^*K}$ to be compact, they are based on the probabilistic assumption of compact recurrence (see Definition \ref{definition:compact_recurrence}). %
In Section~\ref{section:compact_recurrence_of_planar_lambda_steiner_processes}, we verify that planar $\lambda$-Steiner processes are compactly recurrent. In view of the established criteria, the planar version of Question \ref{question:precise_question} is thereby reduced to a  deterministic problem, namely, for which values of $\lambda \in [0,\frac12]$ there exist $S^\lambda$-invariant planar convex bodies. We solve this problem in Section~\ref{section:a_characterization_of_ellipses_and_triangles}, showing that the property of being $S^\lambda$-invariant for some $\lambda\in[0,\frac12]$ characterizes triangles and ellipses. Using this characterization, we are able to fully resolve the planar version of Question \ref{question:precise_question} as follows.
\begin{theorem} \label{theorem:solution_to_planar_question}
    Let $K \in \Ktt$, $\lambda \in [0,\frac12]$ and $G\leq\GL(2,\R)$ be closed. The random $\lambda$-Steiner process starting at $K$ converges modulo $G$ almost surely exactly in the following cases:
    \begin{enumerate}
        \item $\lambda =\frac12$ and $G$ is any closed subgroup of $\GL(2,\R)$, or
        \item $\lambda =0$ and $G$ contains $\SL(2,\R)$.
    \end{enumerate}
    In the first case, the process converges to (the equivalence class modulo $G$ of) the Euclidean ball of volume $\vol K$. In the second case, the process converges to the class of triangles of volume $\vol K$. For all other choices of $\lambda$ and $G$, the process diverges modulo $G$ almost surely.
\end{theorem}

Theorem~\ref{theorem:solution_to_planar_question} admits several interpretations. Broadly speaking, we might say that the natural setting of Steiner symmetrization is Euclidean geometry, whereas shaking belongs into the realm of equiaffine geometry. Moreover, Theorem~\ref{theorem:solution_to_planar_question} suggests that the parallels between Steiner symmetrization and shaking in the plane are best understood as a duality between two exceptional geometric operations; they are exceptional in the sense that their shared properties are not inherited by the family of operations that forms their ``convex hull''. Finally, Theorem~\ref{theorem:solution_to_planar_question} highlights the special role of ellipses and triangles among planar convex bodies. As we show in Section~\ref{section:a_characterization_of_ellipses_and_triangles}, there are no intermediate bodies that would play similar roles for $\lambda$-Steiner symmetrization for $\lambda \in (0,\frac12)$; the absence of such ``attractors'' forces the corresponding random $\lambda$-Steiner processes to diverge almost surely.

The article ends with a section with concluding remarks, addressing some of the questions that have been left open in the preceding sections.

\section{Basic properties of Steiner operations} \label{section:basic_properties_of_steiner_operations}

Throughout the section, let $G$ be a closed subgroup of $\GL(n,\R)$.

\begin{definition}
    We define the \emph{$G$-distance} of $K,L \in \Knn$ by $d_G(K,L)\coloneqq d(K,[L]_G)$.
\end{definition}
Equivalently, the $G$-distance is given by %
\begin{equation} \label{equation:def_g_distance}
        d_G(K,L)=\inf \{\lambda\mu\mid \lambda,\mu\geq 1,\;\exists y,z \in \R^n\colon \exists A \in G\colon \lambda (K-y)\supset A(L-z) \supset \mu^{-1} (K-y)\}.
    \end{equation}
In the extremal case $G=\GL(n,\R)$, the $G$-distance $d_G$ is simply the Banach--Mazur distance $d_{\operatorname{BM}}$. As in the special case of $d_{\operatorname{BM}}$, one readily verifies that $\delta_G\coloneqq\log d_G$ is a metric on the quotient $\Knn/\mathord\sim_G$ for arbitrary closed $G\leq\GL(n,\R)$. The notion of convergence that is induced by this metric is precisely convergence modulo $G$. The other extremal case $d=\dId$ has already been introduced in Section~\ref{section:introduction}. %
It is easy to show that $\delta(K_i,L)$ converges to $0$ if and only if there exists a sequence of vectors $(y_i)_{i\in\N}$ such that $K_i-y_i$ converges to $L$ with respect to the Hausdorff metric $d_{\operatorname H}$. Since the topology of a metric space is fully determined by its sequentially closed sets, this shows that the topology induced by $\delta$ is the restriction of the Hausdorff topology modulo translations to $\Knn$. However, in contrast to $(\Knn, d_{\operatorname H})$, the pseudometric space $(\Knn,\log d)$ is complete, and because polytopes with rational vertices form a dense countable subset, separable. Since any closed $G\leq \GL(n,\R)$ acts isometrically on $(\Knn,\delta)$ and its orbits are closed, the same holds for the pseudometric space $(\Knn,\delta_G)$. This fact will be used in Section~\ref{section:random_convergence_of_generalized_steiner_processes}.

We now formally introduce the notion of a Steiner operation, which has already been discussed in Section~\ref{section:introduction}. %

\begin{definition} \label{definition:steiner_operation}
    A map $\Sigma\colon \Orthogonal(n) \times \Knn \rightarrow \Knn$, $(U,K)\mapsto \Sigma_U(K)$ is called a \emph{Steiner operation} if it has the following properties for all $U \in \Orthogonal(n)$:
    \begin{enumerate}[label=(\roman*)]
        \item \label{item:def_steiner_operation:translation_class_preservation} if $K$ is a translate of $L$, then $\Sigma_UK$ is a translate of $\Sigma_UL$ (translation-class preservation),
        \item \label{item:def_steiner_operation:monotonicity} $K \subset L$ implies $\Sigma_U K \subset \Sigma_UL+z$ for some $z \in \R^n$ (monotonicity),
        \item \label{item:def_steiner_operation:homogeneity} $\Sigma_U\lambda K$ is a translate of $\lambda\Sigma_UK$ for all $\lambda > 0$ (positive homogeneity),
        \item \label{item:def_steiner_operation:rotation_equivariance}$\Sigma_{OU}OK$ is a translate of $O\Sigma_UK$ for all $O \in \SO(n)$ (rotation-equivariance).
    \end{enumerate}
\end{definition}

In this section, we study some basic metric and topological properties of Steiner operations, providing the deterministic foundations for our results on convergence and divergence of random $\Sigma$-processes in the next section. 

Before we proceed, let us briefly discuss how the notion of a Steiner operation generalizes $\lambda$-Steiner symmetrization. Let $\lambda \in [0,\frac12]$ and let $\pi_1\colon \Orthogonal(n)\rightarrow \sphere$ be the projection onto the first column. Then the map $\Sigma^\lambda$ given by $\Sigma^\lambda_{U}K\coloneqq S^\lambda_{\pi_1(U)}K$ is a Steiner operation: Properties \ref{item:def_steiner_operation:translation_class_preservation}-\ref{item:def_steiner_operation:homogeneity} are immediate from Definition~\ref{definition:lambda_steiner_symmetrization}, and property~\ref{item:def_steiner_operation:rotation_equivariance} is also easy to see (see Lemma~\ref{lemma:lambda_steiner_affine_transformation} for a proof). Since $\pi_1$ is a quotient map, $u \mapsto \Sulambda K$ is continuous if and only if $U\mapsto \Sigma^\lambda_UK$ is continuous (for a given $K \in \Knn$). It is therefore natural to identify $S^\lambda$ and $\Sigma^\lambda$ and to say that $S^\lambda$ \emph{is} a Steiner operation. Moreover, the pushforward of the Haar probability measure on $\Orthogonal(n)$ is exactly the rotation-invariant probability measure on $\sphere$, so the random $\Sigma^\lambda$ process and the random $\lambda$-Steiner process starting at $K\in \Knn$ have the same distribution. In particular, our criteria for random convergence of generalized Steiner processes are applicable to random $\lambda$-Steiner processes.

\begin{remark} \label{remark:examples_of_steiner_operations}
    Formally, $\lambda$-Steiner symmetrization is a map $\sphere \times \Knn \rightarrow \Knn$, or equivalently, a map $\tilde{\mathcal{G}}(n,1)\times \Knn \rightarrow \Knn$, where $\tilde{\mathcal{G}}(n,1)$ denotes the oriented Grassmannian, consisting of all oriented one-dimensional subspaces of $\R^n$. The projection $\pi_1 \colon \Orthogonal(n)\rightarrow \sphere$ describes how $\tilde{\mathcal{G}}(n,1)$ can be constructed as a quotient of $\Orthogonal(n)$, which allows us to identify $S^\lambda$ with a Steiner operation. In the case $\lambda=\frac12$, we can forget the orientation, and work with the usual Grassmannian $\mathcal G(n,1)$ of one-dimensional subspaces of $\R^n$. But there are also many other well-studied symmetrization operations in convex geometry, which can be described as maps $\mathcal{G}(n,k)\times \Knn \rightarrow \Knn$ for $k \in \{0,\dots,n\}$. In the following, we collect a non-exhaustive list of examples that can be identified with Steiner operations, which entails that our criteria for random convergence are applicable to these operations as well.
    \begin{enumerate}
        \item Given a $k$-dimensional linear subspace $H$ of $\R^n$, the \emph{Schwarz symmetral} $S_HK$ of $K$ with respect to $H$ is obtained by replacing each non-empty slice $K\cap(H^\bot+x)$, $x\in H$, by a Euclidean ball with same $(n-k)$-dimensional volume and center in $H$. If $H=u^\bot$ is a hyperplane, then the Schwarz symmetral $S_HK$ coincides with the Steiner symmetral $S_uK$. Denoting the projection $\Orthogonal(n) \rightarrow \sphere$ onto the $i$-th column by $\pi_i$, the map
        \begin{equation}
            \Pi\colon\Orthogonal(n) \rightarrow \mathcal G(n,k), \quad U \mapsto \linh\{\pi_1(U),\dots,\pi_k(U)\}
        \end{equation}
        describes how $\mathcal G(n,k)$ can be constructed as a quotient of $\Orthogonal(n)$ and it is straightforward to verify that the map $\Sigma$ given by $\Sigma_{U}K\coloneqq S_{\Pi(U)}K$ is a Steiner operation. By a similar identification as above, it makes sense to say that Schwarz symmetrization \emph{is} a Steiner operation.
        \item For a $k$-dimensional linear subspace $H$ of $\R^n$, the (generalized) \emph{Minkowski symmetral} is given by
        \begin{equation}
            M_HK=\frac12 K + \frac12 \rho_H(K),
        \end{equation}
        where $\rho_H$ denotes the reflection across $H$. Again, one readily verifies that this operation can be identified with a Steiner operation.
        \item \emph{The fiber symmetral} $F_HK$ with respect to $H \in \mathcal G(n,k)$ is obtained  by replacing each non-empty slice $K\cap(H^\bot+x)$, $x\in H$, by the fiber-wise Minkowski average
        \begin{equation}
            \frac12 [K\cap(H^\bot+x)] + \frac12 [\rho_H(K)\cap(H^\bot+x)].
        \end{equation}
        This operation is based on McMullen's notion of the \emph{fiber sum} \cite{mcm99}. As for the previous examples, it is straightforward to verify that fiber symmetrization corresponds to a Steiner operation.
    \end{enumerate} 
\end{remark}

It is convenient to interpret iterated applications of $\Sigma$ as an action of the free semigroup $\Orthogonal(n)^+$ on $\Knn$ and to write $\Sigma_{\mathbf U\mathbf O}K\coloneqq\Sigma_{O_\ell}\cdots\Sigma_{O_1} \Sigma_{U_k}\cdots\Sigma_{U_1}K$ if $\mathbf U=(U_1,\dots,U_k)\in\Orthogonal(n)^k$ and $\mathbf O=(O_1,\dots,O_\ell)\in\Orthogonal(n)^\ell$. In other words, words over $\Orthogonal(n)$ act on $\Knn$ and the group structure on $\Orthogonal(n)$ is ignored for the purposes of this action.

The first lemma of this section explains the role of the properties \ref{item:def_steiner_operation:translation_class_preservation}, \ref{item:def_steiner_operation:monotonicity} and \ref{item:def_steiner_operation:homogeneity} in Definition~\ref{definition:steiner_operation}; these properties ensure that Steiner operations are compatible with our notions of distance on $\Knn$.
\begin{lemma}\label{lemma:steiner_operations_are_non-expansive}
    Every Steiner operation $\Sigma$ is non-expansive for the geometric distance $\dId$, i.e., we have
    \begin{equation}
        \dId(\Sigma_UK,\Sigma_UL) \leq \dId(K,L) \quad \text{for all }K,L \in \Knn\text{ and } U \in \Orthogonal(n).
    \end{equation}
    If $L^*$ is $\Sigma$-invariant modulo $G$, then
    \begin{equation}
        d_G(\Sigma_{\mathbf U}K,L^*)\leq d(K,L^*) \quad \text{for all }K \in \Knn \text{ and } \mathbf U \in \Orthogonal(n)^+.
    \end{equation}
\end{lemma}

\begin{proof}
    Let $U \in \Orthogonal(n)$ and let $\lambda,\mu \geq 1$ and $y_1,y_2 \in \R^n$ such that
    \begin{equation}
        \lambda (K-y_1)\supset L-y_2 \supset \mu^{-1} (K-y_1).
    \end{equation}
    Now monotonicity, homogeneity and the translation-class preservation property of $\Sigma$ imply the existence of $z_1,z_2,z_3 \in \R^n$ such that
    \begin{equation}
        \lambda \Sigma_UK-z_1\supset \Sigma_UL -z_2 \supset \mu^{-1} \Sigma_UK-z_3.
    \end{equation}
    If $\lambda\mu>1$, then translation by $(1-\frac1{1-(\lambda\mu)^{-1}})z_1+\frac1{1-(\lambda\mu)^{-1}}z_3$ yields an inclusion as in \eqref{equation:def_g_distance}, and for $\lambda\mu=1$ there is nothing to show. It follows that $\Sigma_U$ is non-expansive for $\dId$.

    For the second claim, let $L^*$ be $\Sigma$-invariant modulo $G$ and let $\mathbf U \in \Orthogonal(n)^+$. Then there exist $A \in G$ and $z \in \R^n$ with $\Sigma_{\mathbf U}L^*=AL^*+z$ and we get
    \begin{equation}
        d_G(\Sigma_{\mathbf U}K,L^*)\leq \dId(\Sigma_{\mathbf U}K,AL^*+z)=\dId(\Sigma_{\mathbf U}K,\Sigma_{\mathbf U}L^*)\leq \dId(K,L^*). \qedhere
    \end{equation}
\end{proof}

Our arguments in Section~\ref{section:random_convergence_of_generalized_steiner_processes} will be mostly topological. The following lemma provides the first useful topological fact about Steiner operations.

\begin{lemma} \label{lemma:steiner_operations_continuity}
    Let $\Sigma$ be a Steiner operation. Then $L \mapsto \Sigma_U L$ and $O \mapsto \Sigma_O K$ are both continuous for all $K \in \Knn$ and $U \in \Orthogonal(n)$.
\end{lemma}

\begin{proof}
    By Lemma~\ref{lemma:steiner_operations_are_non-expansive} $L \mapsto \Sigma_U L$ is non-expansive and hence continuous. For the second claim, let $(U_i)_{i \in \N}$ be a sequence in $\Orthogonal(n)$ that converges to $U$. Without loss of generality, we assume that $\det U_i=\det U$ for all $i \in \N$. Let $O_i\coloneqq U_iU^{-1}\in\SO(n,\R)$.
    Then  the rotation-equivariance of $\Sigma$ yields
    \begin{align}
        \deltaId(\Sigma_UK,\Sigma_{U_i}K) =\deltaId(\Sigma_UK,\Sigma_{O_iU}O_i O_i^{-1}K)=\deltaId(\Sigma_UK,O_i\Sigma_UO_i^{-1}K)
    \end{align}
    and by the continuity of $L \mapsto \Sigma_U L$ and the fact that $O_i \goestoinfty{i} I_n$, we get
    \begin{align}
        \deltaId(\Sigma_UK,O_i\Sigma_UO_i^{-1}K)&\leq \deltaId(\Sigma_UK,O_i\Sigma_UK)+\deltaId(O_i\Sigma_UK,O_i\Sigma_UO_i^{-1}K)\\
        &=\deltaId(\Sigma_UK,O_i\Sigma_UK)+\deltaId(\Sigma_UK,\Sigma_UO_i^{-1}K)\goestoinfty{i}0. \qedhere
    \end{align}
\end{proof}

The notion of a $\Sigma$-invariant convex body was introduced in Definition \ref{definition:sigma_invariant}. Given a Steiner operation $\Sigma$ on $\Knn$, we write $\mathcal I_G(\Sigma)$ for the set of convex bodies in $\R^n$ that are $\Sigma$-invariant modulo $G$.

\begin{lemma} \label{lemma:set_of_invariant_bodies_is_closed}
    For any Steiner operation $\Sigma$, the set $\mathcal I_G(\Sigma)$ is closed.
\end{lemma}
\begin{proof}
    For $i \in \N,$ let $L_i^*\in \Knn$ be $\Sigma$-invariant modulo $G$ and let $\deltaId(K,L_i^*) \goestoinfty{i}0$. Let $\mathbf U \in \Orthogonal(n)^+$. Then there exists a sequence of matrices $A_i \in G$ and vectors $z_i \in\R^n$ with $\Sigma_{\mathbf U}L_i^*=A_iL_i^*+z_i$ for $i \in \N$. Since $L_i^* \goestoinfty{i} K$ and $A_iL_i^*\goestoinfty{i} \Sigma_{\mathbf U} K$, the operator norms $\norm{A_i}$ and $\norm{A_i^{-1}}$ are bounded and there exists a convergent subsequence $(A_{k(i)})_{i \in \N}$. Since $G$ is closed, $A^* \coloneqq \lim_{i \rightarrow \infty} A_{k(i)}$ is an element of $G$ and we have
    \begin{equation}
        A^*K \underset{i \rightarrow \infty}{\longleftarrow}A_{k(i)}L_{k(i)}^*\goestoinfty{i}\Sigma_{\mathbf U}K,
    \end{equation}
    which shows that $K$ is $\Sigma$-invariant as well.
\end{proof}

The following result clarifies the relation between $\Sigma$-universality and $\Sigma$-invariance.

\begin{lemma} \label{lemma:sigma_invariance_and_sigma_universality}
    Let $\Sigma$ be a Steiner operation and let $K\in\Knn$ be $\Sigma$-universal. If there exists a convex body $L^*\in \Knn$ that is $\Sigma$-invariant modulo $G$, then we have $rK \in [L^*]_G$ for some $r>0$ and $K$ is also $\Sigma$-invariant modulo $G$.
\end{lemma}

\begin{proof}
    Because $K$ is $\Sigma$-universal, there exist a sequence of matrices $U_1,U_2,\dots\in\Orthogonal(n)$ and a number $r>0$ with $\deltaId(\Sigma_{U_i}\cdots\Sigma_{U_1}L^*,rK)\goestoinfty{i}0$. Since $L^*$ is $\Sigma$-invariant modulo $G$, there exist two sequences $A_1,A_2,\dots\in G$ and $z_1,z_2,\dots \in \R^n$ with $A_iL^*+z_i=\Sigma_{U_i}\cdots\Sigma_{U_1}L^*$ for $i\in\N$. Because $\deltaId(A_iL^*,rK)=\deltaId(\Sigma_{U_i}\cdots\Sigma_{U_1}L^*,rK)\goestoinfty{i}0$, the operator norms $\norm{A_i}$ and $\norm{A_i^{-1}}$ are bounded and there exists a convergent subsequence $(A_{k(i)})_{i \in \N}$ with limit $A^* \in G$. Since $\deltaId(A^*L^*,rK)=0$, we have $rK \in [L^*]_G$. To show the second claim, let $\mathbf V \in \Orthogonal(n)^+$. Again, because $L^*$ is $\Sigma$-invariant modulo $G$, there exist two sequences $B_1,B_2,\dots\in G$ and $y_1,y_2,\dots\in\R^n$ such that $B_iL^*+y_i=\Sigma_{\mathbf V}\Sigma_{U_i}\cdots\Sigma_{U_1}L^*$. By Lemma~\ref{lemma:steiner_operations_are_non-expansive}, we have
    \begin{equation}
        \deltaId(B_iL^*,\Sigma_{\mathbf V}rK)=\deltaId(\Sigma_{\mathbf V}\Sigma_{U_i}\cdots\Sigma_{U_1}L^*,\Sigma_{\mathbf V}rK)\leq \deltaId(\Sigma_{U_i}\cdots\Sigma_{U_1}L^*,rK)\goestoinfty{i}0.
    \end{equation}
    By the same argument as above, there exists a convergent subsequence $(B_{k(i)})_{i \in \N}$ with limit $B^* \in G$. Now $\deltaId(B^*L^*,\Sigma_{\mathbf V}rK)=0$ implies $\Sigma_{\mathbf V}rK \in [L^*]_G=[rK]_G$, which shows that $rK$ is $\Sigma$-invariant modulo $G$. By the homogeneity of $\Sigma$, $K$ is also $\Sigma$-invariant modulo $G$.
\end{proof}

The following two notions will be used to formulate our criteria for almost sure convergence and divergence of random $\Sigma$-processes.

\begin{definition}
    Let $\Sigma$ be a Steiner operation. A convex body $K \in \Knn$ is called \emph{$\Sigma$-errant} modulo $G$ if $\deltaId(\Sigma^*K,\mathcal I_G(\Sigma))>0$. The body $K$ is called $\Sigma$-\emph{tethered} modulo $G$ if $\Sigma^*K$ does not contain a $\Sigma$-errant body, i.e., if
    \begin{equation}
        \deltaId(\Sigma^*L,\mathcal I_G(\Sigma))=0 \quad \text{for all } L\in\Sigma^*K.
    \end{equation}
    We write $\mathcal E_G(\Sigma)$ and $\mathcal T_G(\Sigma)$ for the sets of convex bodies in $\R^n$ that are $\Sigma$-errant and $\Sigma$-tethered modulo $G$, respectively.
\end{definition}

We end the section with two further useful topological facts.

\begin{lemma} \label{lemma:sets_of_errant_and_tethered_bodies}
    For any Steiner operation $\Sigma$, the set $\mathcal E_G(\Sigma)$ is open and the set $\mathcal T_G(\Sigma)$ is closed.
\end{lemma}

\begin{proof}
    The first claim follows from the fact that $K \mapsto \deltaId(\Sigma^*K,\mathcal I_G(\Sigma))$ is continuous.

    For the second claim, let $(K_i)_{i \in \N}$ be a sequence of $\Sigma$-tethered convex bodies that converges to $K$ and let $L\in\Sigma^*K$. We have to show that for every $\varepsilon>0$ there exists $\tilde L \in \Sigma^*L$ with $\deltaId(\tilde L, \mathcal I_G(\Sigma))<\varepsilon$. By definition, there exists a finite sequence $\mathbf U \in \Orthogonal(n)^+$ with $L=\Sigma_\mathbf U K$. Let $i \in \N$ with $\deltaId(K,K_i)<\frac\varepsilon2$. Because $K_i$ is $\Sigma$-tethered, there exists a finite sequence $\mathbf V\in \Orthogonal(n)^+$ with $\deltaId(\Sigma_{\mathbf U \mathbf V} K_i, \mathcal I_G(\Sigma))<\frac\varepsilon2$. By Lemma~\ref{lemma:steiner_operations_are_non-expansive}, we have
    \begin{equation}
        \deltaId(\Sigma_{\mathbf U \mathbf V}K,\mathcal I_G(\Sigma)) \leq \deltaId(\Sigma_{\mathbf U \mathbf V} K_i, \mathcal I_G(\Sigma))+\deltaId(\Sigma_{\mathbf U \mathbf V} K_i, \Sigma_{\mathbf U \mathbf V} K) < \frac\varepsilon2 +\deltaId(K_i,K)<\varepsilon. \qedhere
    \end{equation}
\end{proof}

\section{Convergence criteria for random generalized Steiner processes}
\label{section:random_convergence_of_generalized_steiner_processes}

In this section, we state and prove the criteria that generalize Theorem~\ref{theorem:sufficient_condition_for_convergence_simple_version} and Theorem~\ref{theorem:sufficient_condition_for_divergence_simple_version}. Again, let $G$ be a closed subgroup of $\GL(n,\R)$ throughout the section.

\begin{definition} \label{definition:random_sigma_process_starting_at_K}
    Let $\Sigma$ be a Steiner operation and let $K\in \Knn$. The \emph{random $\Sigma$-process starting at $K$} is the Markov process in discrete time given by%
    \begin{equation}
        K_0=K \quad \text{and} \quad K_i = \Sigma_{U_i}K_{i-1} \quad \text{for }i\in\N,
    \end{equation}
    where the matrices $U_i$ are independently sampled from $\mu$, the Haar probability measure on $\Orthogonal(n)$.
\end{definition}

Since $(\Knn,\log d_G)$ is complete and separable, convergence of a sequence is equivalent to being a Cauchy sequence, which implies that the event ``$(K_i)_{i\in\N}$ converges'' is measurable. Most statements about probabilities in this section assert that the respective probability is either 1 or 0. Therefore, they also hold for more general random $\Sigma$-processes, provided that the underlying probability measure on $\Orthogonal(n)^{\N}$ is absolutely continuous with respect to $\mu^{\otimes \N}$; one way to verify the latter for product measures is to employ Kakutani's theorem.

\begin{definition} \label{definition:compact_recurrence}
    We say that a Steiner operation $\Sigma$ is \emph{compactly recurrent for $K \in \Knn$} if there exists a compact subset $C(\Sigma,K)$ of $(\Knn,\deltaId)$ such that the random $\Sigma$-process $(K_t)_{t \in \N_0}$ satisfies
    \begin{equation}
        \Prob [K_t \in C(\Sigma,K) \text{ for infinitely many } t\in \N_0] =1.
    \end{equation}
    We say that $\Sigma$ is \emph{compactly recurrent} if it is compactly recurrent for all $K\in \Knn$.
\end{definition}

In order to make sense of Definition~\ref{definition:compact_recurrence}, we clarify the meaning of compactness in our setting.

\begin{lemma} \label{lemma:characterization_of_compactness}
    A subset $C$ of $(\Knn,\deltaId)$ is compact if and only if $C$ is closed, $\sup_{K \in C}\diam K<\infty$ and $\inf_{K \in C}\vol K>0$.
\end{lemma}

\begin{proof}
    A subset of  the metric space $(\Knn,\deltaId)$ is compact if and only if it is sequentially compact. If $\sup_{K \in C}\diam K=\infty$ or $\inf_{K \in C}\vol K=0$, then we can find a sequence $(L_i)_{i \in \N}$ where $(\diam L_i)_{i\in \N}$ diverges to infinity or $(\vol L_i)_{i\in \N}$ converges to 0. Recalling that convergence in $(\Knn,\log \dId)$ is equivalent to convergence in the Hausdorff distance modulo translations, it follows that such a sequence does not have a convergent subsequence. For the converse statement, let $(L_i)_{i \in \N}$ be a sequence in $C$. Denoting the centroid of $L_i$ by $c_{L_i}$ and the space of compact convex sets by $\mathcal K^n$, the sequence $(L_i-c_{L_i})_{i \in \N}$ is bounded in $(\mathcal K^n,d_H)$. Now the Blaschke selection theorem implies that there exists a subsequence $(L_{k(i)}-c_{L_{k(i)}})_{i \in \N}$ that converges to a body $L^*\in \mathcal{K}^n$ in the metric $d_H$. In particular, the subsequence converges to $L^*$ modulo translations, and hence also in the metric $\deltaId$. Moreover, because volume is continuous with respect to $d_H$, we have $L^*\in \Knn$, and because $C$ is closed, $L^*\in C$.
\end{proof}

Having introduced the notion of compact recurrence, we can now state our characterization of almost sure convergence of generalized Steiner processes.

\begin{theorem} \label{theorem:characterization_tethered}
    If $\Sigma$ is compactly recurrent for $K \in \Knn$, then the random $\Sigma$-process starting at $K$ converges modulo $G$ almost surely if and only if $K$ is $\Sigma$-tethered modulo $G$. If the process converges almost surely modulo $G$, then its limit is almost surely $\Sigma$-invariant modulo $G$.
\end{theorem}

Intuitively speaking, the idea behind Theorem~\ref{theorem:characterization_tethered} is as follows: if a random $\Sigma$-process starts at a $\Sigma$-tethered convex body $K \in \Knn$, then it can always make progress toward the class $\mathcal I_G(\Sigma)$ of invariant bodies, the possibility of progress will be realized eventually, and the progress will never be reversed because $\Sigma$ does not increase the distance to this class by Lemma~\ref{lemma:steiner_operations_are_non-expansive}. Compact recurrence implies that if $(K_i)_{i\in\N}$ converges to $\mathcal I_G(\Sigma)$, then modulo $G$ it converges to an element of $\mathcal I_G(\Sigma)$.  This explains the sufficiency of $K$ being $\Sigma$-tethered; the necessity will be deduced from the following sufficient condition for almost sure divergence.

\begin{theorem} \label{theorem:sufficient_condition_errant}
    If $\Sigma$ is compactly recurrent for $K \in \Knn$, then the random $\Sigma$-process starting at $K$ diverges modulo $G$ almost surely if $K$ is $\Sigma$-errant modulo $G$.
\end{theorem}

The intuitive idea behind Theorem~\ref{theorem:sufficient_condition_errant} is complementary to the one behind Theorem~\ref{theorem:characterization_tethered}: if a random $\Sigma$-process starts at or reaches a $\Sigma$-errant convex body $K\in \Knn$, then it has no chance of arriving at the class $\mathcal I_G(\Sigma)$, and will wander through $\Knn$ until the ends of time.

Before we proceed to prove Theorem~\ref{theorem:characterization_tethered} and Theorem~\ref{theorem:sufficient_condition_errant}, we need a further technical lemma.

\begin{lemma} \label{lemma:transition_kernel_feller_property}
    Let $\Sigma$ be a Steiner operation and let $(K_t)_{t\in \N_0}$ be the random $\Sigma$-process starting at $K$. For any $k \in \N$, the $k$-step transition kernel
    \begin{equation}
        \kappa^k(L,A) \coloneqq \Prob[K_{t+k} \in A \mid K_t=L]
    \end{equation}
    has the Feller property, i.e., for every $f \in C_b(\Knn)$ the following holds: If $(K_i)_{i \in \N}$ converges to $K$ in $\Knn$, then
    \begin{equation}
        \int f \dint\kappa^k(K_i,\cdot) \goestoinfty{i} \int f \dint\kappa^k(K,\cdot).
    \end{equation}
\end{lemma}
\begin{proof}
    We have
    \begin{align}
        \int f \dint\kappa^k(K_i,\cdot)=\int_{\Orthogonal(n)^k} f\left(\Sigma_{\mathbf U}K_i\right) \dint \mu^{\otimes k}( \mathbf U) \goestoinfty{i}\int_{\Orthogonal(n)^k} f\left(\Sigma_{\mathbf U} K\right) \dint\mu^{\otimes k}( \mathbf U)= \int f \dint\kappa^k(K,\cdot),
    \end{align}
    where the second step follows from Lebesgue's dominated convergence theorem, using that $f \circ \Sigma_{\mathbf U}$ is continuous and that $f$ is bounded.
\end{proof}

Because Theorem~\ref{theorem:sufficient_condition_errant} is used in the proof of Theorem~\ref{theorem:characterization_tethered}, it makes sense to prove it first.

\begin{proof}[Proof of Theorem~\ref{theorem:sufficient_condition_errant}]
    Let $K$ be $\Sigma$-errant and let $(K_t)_{t\in \N_0}$ be the random $\Sigma$-process starting at $K$. Let $C(\Sigma,K)$ be as in Definition \ref{definition:compact_recurrence}. Since $L \mapsto \Sigma_\mathbf{U}L$ is continuous for every $\mathbf U \in \Orthogonal(n)^+$ by Lemma~\ref{lemma:steiner_operations_continuity}, the set
    \begin{equation}
        M_k \coloneqq  \left\{L \in \Knn \MID \exists U_1,\dots,U_k\in \Orthogonal(n)\colon \textstyle\max_{i \in [k]}\delta_G(\Sigma_{U_i}\dots \Sigma_{U_1}L,L)>0\right\}
    \end{equation}
    is open for every $k \in \N$. Moreover, every $L \in \Knn$ that is not $\Sigma$-invariant modulo $G$ is contained in $M_k$ for some $k \in \N$ and hence also in $M_j$ for all $j\geq k$. Since $K$ is $\Sigma$-errant, every $L \in \overline{\Sigma^*K}$ is not $\Sigma$-invariant, and since
    \begin{equation}
        S\coloneqq C(\Sigma,K) \cap \overline{\Sigma^*K}
    \end{equation}
    is compact, it follows that there exists a $k^* \in \N$ such that $S\subset M_{k^*}$. We define a sequence of stopping times via $\tau_0=0$ and
    \begin{equation} \label{eq:definition_stopping_times_tau}
        \tau_i \coloneqq \inf\{t\in \N\mid K_t \in C(\Sigma,K), t\geq \tau_{i-1}+k^*\} \quad \text{ for i} \in \N,
    \end{equation}
    where $\inf \emptyset =\infty$. By the defining assumption about $C(\Sigma,K)$, we have
    \begin{equation}
        \Prob(\tau_i<\infty \text{ for all } i\in \N)=1.
    \end{equation}
        For $L \in \Knn$, we consider the continuous function
    \begin{equation}
        f_L\colon \Orthogonal(n)^{k^*} \rightarrow \R ,\quad (U_1,\dots,U_{k^*}) \mapsto \textstyle\max_{i \in [k*]}\delta_G(\Sigma_{U_i}\dots \Sigma_{U_1}L,L),
    \end{equation}
    and set
    \begin{equation}
        \eta(L) \coloneqq \sup \{\eta\geq 0 \mid \mu^{\otimes k^*}(f_L\geq \eta)\geq \eta\}.
    \end{equation}
    We note that $\eta(L)>0$ for every $L \in S\subset M_{k^*}$, because $f_L$ is continuous and there exists a $\mathbf U\in \Orthogonal(n)^{k^*}$ with $f_L(\mathbf{U})>0$. Moreover, by the non-expansiveness of $\Sigma$ from Lemma~\ref{lemma:steiner_operations_are_non-expansive}, we have
    \begin{equation}
        |\eta(L_1)-\eta(L_2)|\leq \norm{f_{L_1}-f_{L_2}}_\infty\leq 2 \deltaId(L_1,L_2) \quad \text{for all } L_1,L_2\in \Knn,
    \end{equation}
    which shows that $K\mapsto\eta(K)$ is continuous. 
    Using again the compactness of $S$, we obtain that $K\mapsto \eta(K)$ assumes a positive minimum $\eta^*$ on $S$. We denote the event $\max_{j \in [k^*]}\delta_G(K_{\tau_i+j},K_{\tau_i})\geq \eta^*$ by $A_i$ and write $\mathcal F_{\tau_i}$ for the $\sigma$-algebra at the stopping time $\tau_i$, which is defined with respect to the natural filtration $\mathcal F_t=\sigma(U_1,\dots,U_t)$. Because $(K_t)_{t\in\N}$ is a Markov process in discrete time, it satisfies the strong Markov property, and we get
    \begin{equation}
        \Prob[A_i|\mathcal F_{\tau_i}]=\mu^{\otimes k^*}(f_{K_{\tau_i}}\geq \eta^*)\geq \eta^* \quad \text{for all }i \in \N.
    \end{equation}
    Using that $A_j \in \mathcal F_{\tau_i}$ for $j\leq i-1$, this implies
    \begin{equation}
        \Prob\left[A_{\ell+i}^\complement\MID\textstyle \bigcap_{j=0}^{i-1} A^\complement_{\ell+j}\right]\leq 1-\eta^* \quad \text{for all } i,\ell \in \N.
    \end{equation}
    By induction, it follows that
    \begin{equation}
        \Prob\left[\textstyle \bigcap_{j=0}^\infty A^\complement_{\ell+j}\right]\leq \prod_{j=1}^\infty (1-\eta^*) = 0 \quad \text{for all } \ell \in \N.
    \end{equation}
    Since, almost surely, infinitely many of the events $\{A_i\}_{i\in\N}$ occur, the random process $(K_t)_{t \in \N_0}$ is almost surely not a Cauchy sequence.%
\end{proof}

We are now ready to prove our characterization of almost sure convergence. The proof of Theorem~\ref{theorem:characterization_tethered} is inspired by Mani-Levitska's classical proof \cite{ml86} of almost sure convergence of random Steiner processes and Volčič's generalization thereof \cite{v13}, as well as Burchard--Fortier's variant concerning random polarizations \cite{bf13}. In particular, the summability argument in the second part is adapted from the proof of \cite[Thm.~2.1]{bf13}. A related adaptation of Burchard--Fortier's argument to the abstract setting of \emph{symmetrization spaces} is provided by Dekeyser--Van~Schaftingen \cite[Sect.~2]{dvs17}; we refer to Remark~\ref{remark:comparison_dvs17} for a comparison with our approach.%

\begin{proof}[Proof of Theorem~\ref{theorem:characterization_tethered}]
    We begin with the case where $K$ is not $\Sigma$-tethered. By definition, there exists a finite sequence $\mathbf U \in \Orthogonal(n)^k$ such that $\Sigma_\mathbf UK$ is $\Sigma$-errant. We consider the continuous map
    \begin{equation}
        h\colon\Orthogonal(n)^k \rightarrow \Knn, \quad \mathbf V\mapsto \Sigma_\mathbf VK.
    \end{equation}
    In view of Lemma~\ref{lemma:sets_of_errant_and_tethered_bodies} and the fact that $\mathbf U \in h^{-1}(\mathcal E_G(\Sigma))$, the set $M\coloneqq h^{-1}(\mathcal E_G(\Sigma))$ is open and non-empty, so we have $\mu^{\otimes k}(M)>0$. In other words, the probability that $K_k$ is $\Sigma$-errant is positive. Denoting the random $\Sigma$-process starting at $L \in \Knn$ by $(L_t)_{t\in\N_0}$, the fact that $(K_t)_{t\in \N_0}$ is a time-homogeneous Markov process implies that
    \begin{equation}
        \Prob[(K_t)_{t\in\N_0} \text{converges} | K_k=L]=\Prob[(L_t)_{t\in\N_0} \text{converges}]\quad \text{almost surely}.
    \end{equation}
    If $L$ is $\Sigma$-errant and $\Sigma$ is compactly recurrent for $L$, then Theorem~\ref{theorem:sufficient_condition_errant} implies that the right-hand side is zero. Since $\Sigma$ is compactly recurrent for $K$, it is compactly recurrent for $K_k$ almost surely, and we get
    \begin{equation}
        \Prob[(K_t)_{t\in\N_0} \text{diverges}]\geq \Prob[K_k \text{ is $\Sigma$-errant}]=\mu^{\otimes k}(M)>0.
    \end{equation}
    
    Now let $K$ be $\Sigma$-tethered and let $C(\Sigma,K)$ be as in Definition \ref{definition:compact_recurrence}. For $\varepsilon>0$, let
    \begin{equation}
        \mathcal I_\varepsilon \coloneqq \{L \in \Knn \mid \deltaId(L,\mathcal I_G(\Sigma))<\varepsilon\} \quad \text{and} \quad f\coloneqq \min \{ \deltaId(\cdotbox,\mathcal I_G(\Sigma)),2\varepsilon\}.
    \end{equation}
    It follows from Definition~\ref{definition:sigma_invariant} that $L^* \in \mathcal I_G(\Sigma)$ implies $\Sigma^*L^*\subset\mathcal I_G(\Sigma)$. Using the non-expansiveness of $\Sigma$ from Lemma~\ref{lemma:steiner_operations_are_non-expansive}, this shows that $L \in \mathcal I_\varepsilon$ implies $\Sigma^*L\subset \mathcal I_\varepsilon$. Moreover, it shows that $f$ is $\Sigma$-monotone, i.e., $f$ is never increased by applying $\Sigma$.
    
    We will first show that
    \begin{equation} \label{equation:proof_of_characterizatin_intermediate_claim}
        \Prob[K_t \in \mathcal I_\varepsilon \text{ for almost all } t\in \N]=\Prob[K_t \in \mathcal I_\varepsilon \text{ for some } t\in \N]=1.
    \end{equation}
    Because $L \mapsto \Sigma_\mathbf{U}L$ is continuous for every $\mathbf U \in \Orthogonal(n)^+$, the set
    \begin{equation}
        M_k\coloneqq \{L \in \Knn \mid \exists \mathbf U\in \Orthogonal(n)^k\colon \Sigma_\mathbf UL \in \mathcal I_\varepsilon\}
    \end{equation}
    is open for every $k \in \N$. Moreover, every $\Sigma$-tethered body $L \in \Knn$ is contained in $M_k$ for some $k \in \N$ and we have $M_k \subset M_{k+1}$ for all $k \in \N$. Since the set
    \begin{equation}
        S\coloneqq \mathcal T_G(\Sigma) \cap C(\Sigma,K)
    \end{equation}
    is compact by Lemma~\ref{lemma:sets_of_errant_and_tethered_bodies}, there exists a $k^* \in \N$ with $S \subset M_{k^*}$. This shows that for every $L \in S \cap  (\mathcal{I}_\varepsilon)^\complement$, there exists a sequence of $k^*$ matrices $\mathbf U \in \Orthogonal(n)^{k^*}$ with
    \begin{equation} \label{eq:progress_in_k*_steps_is_possible}
        f(L)\geq\varepsilon>f(\Sigma_{\mathbf U} L).
    \end{equation}
    In words, \eqref{eq:progress_in_k*_steps_is_possible} simply says that if $L$ has distance at least $\varepsilon$ to the set of $\Sigma$-invariant bodies modulo $G$, then this distance can be made strictly smaller than $\varepsilon$ by a suitable sequence of $k^*$ applications of $\Sigma$ to $L$. Because
    \begin{equation}
        \Orthogonal(n)^{k^*} \rightarrow \R, \quad \mathbf U \mapsto f(\Sigma_{\mathbf U} L)
    \end{equation}
    is continuous, these considerations imply that
    \begin{equation}
        \phi(L)\coloneqq f(L)-\E_{\mu^{\otimes k^*}}[f(\Sigma_\mathbf UL)]=f(L)-\int_{\mathbb \Orthogonal(n)^{k^*}} f(\Sigma_\mathbf UL) \dint \mu^{\otimes k^*}(\mathbf U)>0 \quad \text{for all }L \in S \cap \mathcal (I_\varepsilon)^\complement.
    \end{equation}
    By Lemma~\ref{lemma:transition_kernel_feller_property}, the $k^*$-step transition kernel has the Feller property, so the map $L \mapsto \phi(L)$ is continuous. If the set $S \cap \mathcal (I_\varepsilon)^\complement$ is empty, then \eqref{equation:proof_of_characterizatin_intermediate_claim} is immediate. Otherwise, the map $L \mapsto \phi(L)$ attains its infimum $\delta^*>0$ on the compact set $S \cap \mathcal (I_\varepsilon)^\complement$.
    
    We define a sequence of stopping times via $\tau_0=0$ and
    \begin{equation} 
        \tau_i \coloneqq \inf\{t\in \N\mid K_t \in C(\Sigma,K), t\geq \tau_{i-1}+k^*\} \quad \text{ for i} \in \N.
    \end{equation}
    Again using the fact that $f$ is $\Sigma$-monotone as well as the strong Markov property, we get
    \begin{equation}
        f(K)\geq\sum_{i=1}^\infty\E[f(K_{\tau_{i}})-f(K_{\tau_{i+1}})]
        \geq\sum_{i=1}^\infty\E[f(K_{\tau_{i}})-f(K_{\tau_{i}+k^*})] =\sum_{i=1}^\infty\E[f(K_{\tau_i})-\E_{\mu^{\otimes k^*}}[f(\Sigma_\mathbf UK_{\tau_i})]].
    \end{equation}
    Now we observe that $K_{\tau_i}\in S$ almost surely, by the definition of $C(\Sigma,K)$ and the fact that every $L\in \Sigma^*K$ is $\Sigma$-tethered. By the definition of $\delta^*$, it follows that
    \begin{equation}
        f(K)\geq\sum_{i=1}^\infty\E[f(K_{\tau_i})-\E_{\mu^{\otimes k^*}}[f(\Sigma_\mathbf UK_{\tau_i})]] \geq \sum_{i=1}^\infty \Prob[K_{\tau_i} \notin \mathcal I_\varepsilon] \cdot\delta^*,
    \end{equation}
    and because $\delta^*>0$, this shows that $(\Prob[K_{\tau_i} \notin \mathcal I_\varepsilon])_{i \in \N}$ is summable. In particular, $\Prob[K_{\tau_i}\notin \mathcal I_\varepsilon]\goestoinfty{i}0$. Since the events $\{K_{\tau_i}\notin \mathcal I_\varepsilon\}$ are decreasing, this shows \eqref{equation:proof_of_characterizatin_intermediate_claim}, i.e., that $(K_t)_{t \in \N}$ almost surely visits $\mathcal I_\varepsilon$ and never leaves again. Considering $\varepsilon=\frac1m$ for $m \in \N$, we get $\Prob[\deltaId(K_t,\mathcal I_G(\Sigma))\goestoinfty{t}0]=1.$
    
    It remains to show that, for a given path of the Markov process (corresponding to a sequence of matrices $U_i$, $i \in \N$) that visits $C(\Sigma,K)$ infinitely often, the condition $\deltaId(K_{t},\mathcal I_G(\Sigma))\goestoinfty{t}0$ implies that $(K_t)_{t\in \N_0}$ converges to a $\Sigma$-invariant body modulo $G$. %
    Because $C(\Sigma,K)$ is compact, such a path admits a subsequence $(K_{t(i)})_{i \in \N}$ that converges to a body $L^* \in C(\Sigma,K)$. Because $L \mapsto \deltaId(L, \mathcal I_G(\Sigma))$ is continuous, we have $\deltaId(L^
    *, \mathcal I_G(\Sigma))=0$ and hence, by Lemma~\ref{lemma:set_of_invariant_bodies_is_closed}, $L^* \in \mathcal I_G(\Sigma)$. For any $\varepsilon>0$, there exists an $i \in \N$ with $\deltaId(K_{t(i)},L^*)<\varepsilon$. But now Lemma~\ref{lemma:steiner_operations_are_non-expansive} implies that $\delta_G(K_{t(i)+k},L^*)\leq \deltaId(K_{t(i)},L^*)<\varepsilon$ for all $k\in \N$.
\end{proof}

\begin{remark}
    If $\Sigma$ is compactly recurrent and every $K \in \Knn$ is $\Sigma$-tethered modulo $G$, then Theorem~\ref{theorem:characterization_tethered} implies that all $\Sigma$-processes starting anywhere in $\Knn$ converge \emph{simultaneously} modulo $G$ almost surely. As observed by Burchard--Fortier \cite[p.~554]{bf13} in the context of random polarizations, this follows because $\Sigma$ is non-expansive and $\Knn$ is separable.
\end{remark}

\begin{remark} \label{remark:non-evanescence}
    Theorem~\ref{theorem:characterization_tethered} characterizes almost sure convergence of the random $\Sigma$-process starting at $K$, given that $\Sigma$ is compactly recurrent for $K$. Since for a given path $U_1,U_2,\dots$, the sequence $(\Sigma_{U_i}\dots\Sigma_{U_1}K)_{i\in\N}$ can only converge modulo $\{\Id\}$ if it visits a compact set infinitely often, it is tempting to think that compact recurrence (for the initial body $K$) is logically necessary for almost sure convergence modulo $\{\Id\}$. But this does not follow, because the set $C(\Sigma,K)$ from Definition~\ref{definition:compact_recurrence} is required to be independent of the sample path. This line of reasoning suggests to replace the assumption of compact recurrence by the weaker assumption of non-evanescence \cite{mt93}, i.e., the assumption that
    \begin{equation} \label{eq:compact_recurrence}
        \Prob[(K_t)_{t\in \N} \text{ visits every compact set at most finitely often}] =0.
    \end{equation}
    There seems to be no apparent reason why Theorem~\ref{theorem:characterization_tethered} and Theorem~\ref{theorem:sufficient_condition_errant} should fail if compact recurrence is replaced by the weaker assumption \eqref{eq:compact_recurrence}, although some additional technical issues have to be considered. Since this conjectural strengthening is irrelevant for Question~\ref{question:precise_question}, we leave it open at this point.
\end{remark}

As noted by Bianchi--Burchard--Gronchi--Volčič, ``[m]ost constructions of convergent sequences of Steiner symmetrizations rely on auxiliary geometric functionals, such as the perimeter and the moment of inertia, that decrease monotonically along the sequence'' \cite[p.~1695]{bbgv12}. Such auxiliary functionals are also used to establish almost sure convergence by Mani-Levitska \cite{ml86}, Volčič \cite{v13}, Burchard--Fortier \cite{bf13}, and are called \emph{asymmetries} in the abstract framework of Dekeyser--Van~Schaftingen \cite[Sect.~2]{dvs17}. The following proposition shows that the existence of such a functional implies that every $K\in \Knn$ is $\Sigma$-tethered modulo $G$; conversely, if every $K\in \Knn$ is $\Sigma$-tethered modulo $G$, then the distance to $\mathcal I_G(\Sigma)$ is such a functional.

\begin{proposition}
    Let $\Sigma$ be compactly recurrent. We assume the existence of a continuous geometric functional $f \colon \Knn \rightarrow \R$ that is monotone with respect to $\Sigma$ and has the following property: For every $L \in \Knn$ that is not $\Sigma$-invariant modulo $G$, there exists a $\mathbf U \in \Orthogonal(n)^+$ such that $f(\Sigma_{\mathbf U}L)<f(L)$. Then every $K\in \Knn$ is $\Sigma$-tethered modulo $G$.
\end{proposition}

\begin{proof}
    Let $K \in \Knn$ and $L \in \Sigma^*K$. We have to show that $\deltaId(\Sigma^*L,\mathcal I_G(\Sigma))=0$. Let $C(\Sigma,L)$ be as in Definition~\ref{definition:compact_recurrence} and set $m \coloneqq \inf f(\Sigma^*L \cap C(\Sigma,L)) \in \R$. Because $C(\Sigma,L)$ is compact, there exists a convergent sequence $(L_i)_{i \in \N}$ in $\Sigma^*L \cap C(\Sigma,L)$ with limit $L^*$ such that $(f(L_i))_{i \in \N}$ converges to $f(L^*)=m$. We assume towards a contradiction that $L^*$ is not $\Sigma$-invariant modulo $G$. By assumption, there exists a $\mathbf U \in \Orthogonal(n)^+$ with $f(\Sigma_{\mathbf U}L^*)<f(L^*)$. Since $(f(\Sigma_{\mathbf U}L_i))_{i \in \N}$ converges to $f(\Sigma_{\mathbf U}L^*)<m$, there exists an $i \in \N$ with $f(\Sigma_{\mathbf U}L_i)<m$. By the continuity of $f$, it follows that the random $\Sigma$-process starting at $L$ visits the sublevel set $f^{-1}((-\infty,m))$ with positive probability. Because $f$ is monotone with respect to $\Sigma$, it cannot return to $C(\Sigma,L)$ from there, which is a contradiction to the fact that $C(\Sigma,L)$ is visited infinitely often almost surely. It follows that $L^*$ is $\Sigma$-invariant modulo $G$, which implies the claim.    
\end{proof}

\begin{remark} \label{remark:comparison_dvs17}
    In using continuity and non-expansiveness to transfer Burchard--Fortier's summability argument \cite{bf13} to an abstract setting, we follow Dekeyser--Van~Schaftingen's \cite[Sect.~2]{dvs17} study of random convergence of symmetrization processes. In their elegant framework, a \emph{$\star$-compatible symmetrization space} $S$ is a family of continuous, idempotent and non-expansive operations acting on a metric space $X$, where $X$ is equipped with a non-expansive and idempotent map $\star$ that assigns an invariant object $K^\star$ to every $K\in X$ (generalizing the role of the Euclidean ball for Steiner symmetrization). Assuming that the forward orbit of each $K \in X$ under $S$ is precompact, Dekeyser--Van~Schaftingen provide a characterization of simultaneous almost sure convergence of all random $S$-processes starting at the elements of $X$ in terms of the existence of invariant objects that are distinct from $K^\star$ \cite[Thm.~2.8]{dvs17}. This characterization requires that $S$ admits a \emph{strict asymmetry} (i.e., a strictly monotone auxiliary functional) in the sense of \cite[Def.~2.3]{dvs17}; this assumption has to be added to the statement of \cite[Thm.~2.8]{dvs17} and the strict asymmetry has to be substituted for $\mathcal A(K)=d(K,K^\star)$ in the proof. Although their basic framework is related to ours, Dekeyser--Van~Schaftingen are concerned with different technical issues than the ones imposed by Question~\ref{question:precise_question}, and their results are not directly applicable to our problem.
\end{remark}

Based on the criteria that we proved in this section, it is straightforward to derive the special cases from Section~\ref{section:introduction}.

\begin{proof}[Proof of Theorem~\ref{theorem:sufficient_condition_for_convergence_simple_version}]
    If $\overline{\Sigma^*K}$ is compact, then $\Sigma$ is obviously compactly recurrent for $K$. Because $L^*$ is $\Sigma$-invariant modulo $G$, the homogeneity of $\Sigma$ implies that $rL^*\in\mathcal I_G(\Sigma)$ for all $r>0$. Since $L^*$ is $\Sigma$-universal, for every $L\in \Knn$ there exists a number $r_L>0$ with 
    \begin{equation}
        0=\deltaId(\Sigma^*L,r_LL^*)\geq \deltaId(\Sigma^*L,\mathcal I_G(\Sigma)).  
    \end{equation}
    This shows that every $L\in \Knn$ is $\Sigma$-tethered. Now Theorem~\ref{theorem:characterization_tethered} implies that the random $\Sigma$-process starting at $K$ almost surely converges modulo $G$, and that the limit is almost surely $\Sigma$-invariant modulo $G$. By Lemma~\ref{lemma:sigma_invariance_and_sigma_universality}, this implies that the limit is almost surely $G$-equivalent to a homothet of $L^*$.
\end{proof}

\begin{proof}[Proof of Theorem~\ref{theorem:sufficient_condition_for_divergence_simple_version}]
    Again, the compact recurrence of $\Sigma$ for $K$ follows from the compactness of $\overline{\Sigma^*K}$. Since there is no $\Sigma$-invariant body modulo $G$, every $L\in\Knn$ is $\Sigma$-errant modulo $G$, and Theorem~\ref{theorem:sufficient_condition_errant} implies that the random $\Sigma$-process starting at $K$ almost surely diverges modulo $G$.
\end{proof}

In order to apply the criteria that were established in this section to $\lambda$-Steiner symmetrization, we still have to verify that $\lambda$-Steiner symmetrization is compactly recurrent. This will be done in the planar setting in the next section. In the special case $\lambda=\frac12$, however, we already recover Mani-Levitska's result \cite{ml86} (modulo translations).

\begin{corollary} \label{corollary:convergence_to_ball_for_lambda=1/2}
    Let $K\in \Knn$. The random Steiner process starting at $K$ converges almost surely to any Euclidean ball $B\in\Knn$ with $\vol B=\vol K$.
\end{corollary}

\begin{proof}
    Since Steiner symmetrization does not increase diameter and preserves volume, Lemma~\ref{lemma:characterization_of_compactness} implies that $\overline{\Sigma^*K}$ is compact. Now the Euclidean ball is $S^\frac12$-universal and $S^\frac12$-invariant modulo $\{\Id\}$, so the claim follows from Theorem~\ref{theorem:sufficient_condition_for_convergence_simple_version}.
\end{proof}

More generally, we have the following.

\begin{corollary} \label{corollary:convergence_to_ball_for_examples}
    Let $K \in \Knn$ and let $\Sigma$ be a Steiner operation that does not decrease volume. If the Euclidean ball is $\Sigma$-universal and $\Sigma$-invariant modulo $\{\Id\}$, then the random $\Sigma$-process converges almost surely to a Euclidean ball (whose radius might depend on the sample path).
\end{corollary}

\begin{proof}
    Let $R>0$ such that $K \subset RB$, where $B$ is the Euclidean unit ball. By the monotonicity of $\Sigma$ and the fact that $B$ is $\Sigma$-invariant modulo $\{\Id\}$, we have $\Sigma_{\mathbf U}K\subset RB$ up to translation and hence $\diam \Sigma_{\mathbf U}K \leq 2R$ for any $\mathbf U \in \Orthogonal(n)^+$. Since $\Sigma$ does not decrease volume, Lemma~\ref{lemma:characterization_of_compactness} implies that $\overline{\Sigma^*K}$ is compact. Again, the claim follows from Theorem~\ref{theorem:sufficient_condition_for_convergence_simple_version}.
\end{proof}

Corollary~\ref{corollary:convergence_to_ball_for_examples} can be applied to the examples from Remark~\ref{remark:examples_of_steiner_operations}: For any $k \in \{1,\dots,n-1\}$, the Schwarz, Minkowski and fiber symmetrization with respect to $H \in \mathcal{G}(n,k)$ do not decrease volume and any Euclidean ball is invariant modulo $\{\Id\}$ under these operations. Moreover, the Euclidean ball is universal for Schwarz \cite[Thm.~6.9]{bgg22}, Minkowski \cite[Thm.~6.1]{bgg22} and fiber symmetrization \cite[Thm.~6.1, Prop.~6.3]{bgg22}. Therefore, Corollary~\ref{corollary:convergence_to_ball_for_examples} gives a unified proof for the facts that random generalized Steiner processes induced by Schwarz, Minkowski and fiber symmetrization for  $k \in \{1,\dots,n-1\}$ converge to Euclidean balls almost surely.

\section{Compact recurrence of planar \texorpdfstring{$\lambda$}{λ}-Steiner processes} \label{section:compact_recurrence_of_planar_lambda_steiner_processes}

Since Steiner symmetrization does not increase the diameter of a set while preserving its volume, the closure of the forward orbit $\overline{S^*(K)}$ is compact for any $K\in \Knn$, so Proposition~\ref{proposition:planar_lambda_steiner_symmetrization_is_compactly_recurrent} holds trivially for $\lambda=\frac12$. In the case of shaking, however, we encounter a very different scenario. Let $T_0\in \Ktt$ be a triangle and let $t_0>0$ be the maximal length of its sides. Shaking $T_0$ in the direction of a longest side, we obtain a right triangle $T_1$, whose maximal side length $t_1$ is given by $t_1=\sqrt{t_0^2+\frac{4(\vol T_0)^2}{t_0^2}}$. Iterating this process, we obtain a sequence of right triangles $(T_i)_{i\in\N}$, whose maximal side lengths satisfy the recurrence relation
\begin{equation}
    t_i = \sqrt{t_{i-1}^2+\frac{4(\vol T_0)^2}{t_{i-1}^2}}=t_{i-1} \cdot \sqrt{1+\frac{4(\vol T_0)^2}{t_{i-1}^4}} \quad \text{for } i \in \N.
\end{equation}
If $(t_i)_{i\in \N}$ was bounded, we could find a constant $c>1$ with $\sqrt{1+\frac{4(\vol T_0)^2}{t_{i-1}^4}}>c$ for all $i\in \N$, which leads to a contradiction. It follows that the closure of the forward orbit $\overline{B^*T_0}$ is not compact, and by the monotonicity of shaking, we even get that $\overline{B^*K}$ is not compact for any $K\in \Ktt$. Considering iterated prisms over planar convex bodies, this claim can be extended to all dimensions $n \geq 2$. It follows that Theorem~\ref{theorem:sufficient_condition_for_convergence_simple_version} and Theorem~\ref{theorem:sufficient_condition_for_divergence_simple_version} cannot be applied to Question~\ref{question:precise_question} for general values of $\lambda \in [0,\frac12]$; this motivates the assumption of compact recurrence in the criteria from Section~\ref{section:random_convergence_of_generalized_steiner_processes}. In this section, we verify that Theorem~\ref{theorem:characterization_tethered} and Theorem~\ref{theorem:sufficient_condition_errant} can be used to resolve Question~\ref{question:precise_question} in the planar case. In other words, our goal is to prove the following result.

\begin{proposition} \label{proposition:planar_lambda_steiner_symmetrization_is_compactly_recurrent}
    For every $\lambda \in [0,\frac12]$, $\lambda$-Steiner symmetrization is compactly recurrent in $\R^2$.
\end{proposition}

Considering the sequence of right triangles $(T_i)_{i\in\N}$, we observed that the diameter of a $0$-Steiner process can become arbitrarily large. However, as the following lemma shows, the growth rate of the diameter will decrease as the isodiametric quotient $\frac{\vol K}{(\diam K)^2}$ decreases.
\begin{lemma} \label{lemma:bound_on_growth_of_diameter}
    Let $K \in \Ktt$, $\lambda \in [0,\frac12]$ and $u \in\sphereone$. Then
    \begin{equation}
        \diam(\Sulambda K) \leq \sqrt{(\diam K)^2+8\left(\frac{\vol K}{\diam K}\right)^2}.
    \end{equation}
\end{lemma}

\begin{proof}
    Let $x,y \in \Sulambda K$ with $\norm{x-y}=\diam \Sulambda K$. Let
    \begin{equation}
        L_1 \coloneqq K \cap (x+\linh \{u\}), \quad L_2 \coloneqq K \cap (y+\linh \{u\})
    \end{equation}
    be the two chords of $K$ corresponding to the points $x,y$ and assume without loss of generality that $\alpha\coloneqq\vol_1(L_1) \geq \vol_1(L_2)$. Setting $\beta \coloneqq \norm{\pi_{u^\bot}(x-y)}$, we have $\frac{\alpha \beta}{2}\leq \vol K $, and hence
    \begin{equation}
        \diam \Sulambda K = \norm{x-y} \leq \sqrt{\alpha^2+\beta^2}  \leq \sqrt{\max\{\alpha^2,\beta^2\}+\frac{4(\vol K)^2}{\max\{\alpha^2,\beta^2\}}}.
    \end{equation}
    We restrict our attention to the non-trivial case where $\diam \Sulambda K \geq \diam K$. There, we have
    \begin{equation}
        (\diam K)^2 \leq \left(\diam \Sulambda K\right)^2 \leq \alpha^2+\beta^2 \leq 2 \max\{\alpha^2,\beta^2\},
    \end{equation}
    which implies
    \begin{equation}
        \diam \Sulambda K \leq \sqrt{\max\{\alpha^2,\beta^2\}+\frac{4(\vol K)^2}{\max\{\alpha^2,\beta^2\}}} \leq \sqrt{(\diam K)^2+\frac{8(\vol K)^2}{(\diam K)^2}}. \qedhere
    \end{equation}
\end{proof}

Lemma~\ref{lemma:bound_on_growth_of_diameter} yields a deterministic upper bound on the diameter of any $\lambda$-Steiner process.

\begin{lemma} \label{lemma:sum_of_inverse_diameters_diverges}
    Let $K \in \Ktt$, $\lambda \in [0,\frac12]$ and let $u_1,u_2,\ldots \in\sphereone$ be an arbitrary sequence of unit vectors. 
    Then
    \begin{equation} \label{equation:bound_on_diameters}
        \diam S^\lambda_{u_i}\cdots S^\lambda_{u_1}K \leq \diam K + i \sqrt{2\pi\vol K} \quad \text{for all } i\in \N.
    \end{equation}
\end{lemma}
\begin{proof}
    We write $K_0\coloneqq K$ and $K_i\coloneqq S^\lambda_{u_i}\cdots S^\lambda_{u_1}K$ for $i \in \N$. By the isodiametric inequality, we have $(\diam K_i)^2\geq \frac4{\pi} \vol K_i=\frac4{\pi} \vol K$ for all $i \in \N_0$. Using Lemma~\ref{lemma:bound_on_growth_of_diameter}, we get
    \begin{equation}
        \diam K_{i+1} \leq \sqrt{(\diam K_i)^2} + \sqrt{\frac{8(\vol K_i)^2}{(\diam K_i)^2}}  \leq \diam K_i +\sqrt{2\pi\vol K} \quad \text{for all } i \in \N_0,
    \end{equation}
    which implies the claim. %
\end{proof}

We have thus established that the isodiametric coefficient of a planar random $\lambda$-Steiner process cannot shrink too quickly. On the other hand, as the following lemma shows, if the isodiametric coefficient is not too small, then there is a reasonable probability that it will be bounded from below by a constant after applying a random $\lambda$-Steiner symmetrization.

\begin{lemma} \label{lemma:lower_bound_on_diameter_probability}
    Let $\lambda \in [0,\frac12]$, $K \in \Ktt$ and let $u$ be uniformly distributed on $\sphereone$. Then
    \begin{equation}
        \sigma\left(\diam\Sulambda K\leq 5\sqrt{\vol K}\right) \geq \frac2\pi \cdot \frac{\sqrt{\vol K}}{\diam K},
    \end{equation}
    where $\sigma$ denotes the rotation-invariant probability measure on $\sphereone$.
\end{lemma}

\begin{proof}
    By the homogeneity of $S^\lambda$, we can assume without loss of generality that $\vol K=1$. If $D\coloneqq \diam K<2$, then Lemma~\ref{lemma:bound_on_growth_of_diameter} implies $\diam(\Sulambda K)^2 \leq D^2+\frac8{D^2}\leq 4+2\pi<5^2$, so the claim holds for all $u \in \sphereone$. It remains to consider the case $D\geq 2$. Denoting the width of $K$ in the direction of the line $u^\bot$ by $w(K,u^\bot)$, we first show that in this case
    \begin{equation} \label{eq:lemma_lower_bound_on_diameter_probability}
        \Prob(1 \leq w(K,u^\bot) \leq 4) \geq \frac2\pi \cdot \frac1{\diam K}.
    \end{equation}
    For this, let $x,y \in K$ with $D =\norm{x-y} \geq2$. Defining $\theta_u \in [0,\frac\pi2]$ as the angle between the lines $\aff [x,y]$ and $\linh \{u\}$, the random variable $\theta_u$ is uniformly distributed on $[0,\frac\pi2]$. Clearly, we have
    \begin{equation}
        D \cdot \sin (\theta_u) \leq w(K,u^\bot).
    \end{equation}
    With regard to an upper bound on $w(K,u^\bot)$, let $v \in \sphereone$ be orthogonal to $x-y$ and observe that $K$ is contained in the rectangle $Q = [x,y]+[-\frac2D v,\frac2D v]$, due to the assumption that $\vol K=1$. It follows that
    \begin{equation}
        w(K,u^\bot) \leq w(Q,u^\bot)\leq D \cdot \sin(\theta_u) + \frac4D.
    \end{equation}
    Therefore, if $\sin(\theta_u) \in [\frac1D,\frac2D]$, it follows that
    \begin{equation}
        1 \leq w(K,u^\bot) \leq 2+ \frac4D \leq 4,
    \end{equation}
    where the last step uses the assumption $D\geq 2$. 
    Again using $D \geq 2$, we know that there exists a $\theta_u \in [0,\frac\pi2]$ with $\sin(\theta_u) \in [\frac1D,\frac2D]$, and because the sine function is 1-Lipschitz, it follows that the event $\sin(\theta_u) \in [\frac1D,\frac2D]$ occurs with probability at least $\frac{\frac1D}{\frac\pi2}=\frac{2}{\pi D}$, which shows \eqref{eq:lemma_lower_bound_on_diameter_probability}.

    To conclude the proof, we show that for any $u \in \sphereone$, the estimate $1 \leq w(K,u^\bot) \leq 4$ implies that $\diam\Sulambda K\leq 5$, provided that $\vol K =1$. By the latter assumption, we know that any chord $K \cap (x+\linh \{u\})$, $x \in K$, has length at most $\frac2{w(K,u^\bot)}$. Recalling the definition of $\Sulambda K$, we have $\scpr{x-y,u} \leq \frac2{w(K,u^\bot)}$ for all $x,y \in\Sulambda K$ and hence, under the assumption that $1 \leq w(K,u^\bot) \leq 4$,
    \begin{equation}
        \diam \Sulambda K \leq \sqrt{w(K,u^\bot)^2+\frac4{w(K,u^\bot)^2}}\leq \sqrt{16+4} < 5. \qedhere
    \end{equation}
\end{proof}

Combining the previous lemmas with a standard argument in the spirit of the second Borel--Cantelli lemma, we obtain the following statement.

\begin{lemma} \label{lemma:probability_that_infinitely_many_events_occur_is_1}
    Let $K \in \Ktt$, $\lambda \in [0,\frac12]$ and let $u_1,u_2,\dots$ be a sequence of independent random vectors that are uniformly distributed on $\mathbb{S}^{1}$. For the random $\lambda$-Steiner process starting at K, given by
    \begin{equation}
        K_i \coloneqq (S_{u_i}^\lambda \circ S_{u_{i-1}}^\lambda \circ \dots \circ S_{u_{1}}^\lambda) (K), \quad i \in \N,
    \end{equation}
    the probability that infinitely many of the events $\{\diam K_i \leq 5 \sqrt{\vol K}\}$ occur is 1.
\end{lemma}

\begin{proof}%
    Writing $A_i \coloneqq \{\diam K_i >5\sqrt{\vol K}\}$ for $i\in \N$, it suffices to show that
    \begin{equation}
        \Prob\left[\bigcap_{i=N}^\infty A_i\right] = 0 \quad \text{for all }N\in\N.
    \end{equation}
    Let $N \in \N$. If there exists an index $M \geq N$ with $\Prob\left[\bigcap_{i=N}^M A_i\right] = 0$, there is nothing to show. Otherwise, we have
    \begin{align} \label{eq:lemma_probability_that_infinitely_many_events_occur_is_1}
        \Prob\left[\bigcap_{i=N}^\infty A_i\right]= \prod_{j=N}^\infty \frac{\Prob\left[\bigcap_{i=N}^jA_i\right]}{\Prob\left[\bigcap_{i=N}^{j-1}A_i\right]} = \prod_{j=N}^\infty \Prob \left[A_j \MID \bigcap_{i=N}^{j-1}A_i\right].
    \end{align}
    We consider a single factor of the infinite product. Writing $\nu$ for the conditional distribution of $K_{j-1}$, given that the event $\bigcap_{i=N}^{j-1}A_i$ occurred, the Markov property implies that
    \begin{equation}
        \Prob \left[A_j \MID \bigcap_{i=N}^{j-1}A_i\right]= \int\int_{\sphereone}\One_{\{\diam S_u^\lambda(K_{j-1}) > 5\sqrt{\vol(K)}\}} \dint \sigma(u) \dint\nu(K_{j-1}).
    \end{equation}
    By Lemma~\ref{lemma:lower_bound_on_diameter_probability} and Lemma~\ref{lemma:sum_of_inverse_diameters_diverges}, we have
    \begin{equation}
        \Prob \left[A_j \MID \bigcap_{i=N}^{j-1}A_i\right]\leq \int\left(1-\frac2\pi \cdot \frac{\sqrt{\vol K_{j-1}}}{\diam K_{j-1}}   \right)\dint\nu(K_{j-1})\leq 1-\frac2\pi \frac{\sqrt{\vol K}}{\diam K + (j-1) \sqrt{2\pi\vol K}}.
    \end{equation}
    Plugging this back into \eqref{eq:lemma_probability_that_infinitely_many_events_occur_is_1}, we get
    \begin{align}
        \Prob\left[\bigcap_{i=N}^\infty A_i\right]&\leq \prod_{j=N}^\infty \left(1-\frac2\pi  \frac{\sqrt{\vol K}}{\diam K + (j-1) \sqrt{2\pi\vol K}}\right).
    \end{align}
    Since the infinite product on the right-hand side is zero, the claim follows.
\end{proof}

It is now easy to prove the main result of this section.

\begin{proof}[Proof of Proposition~\ref{proposition:planar_lambda_steiner_symmetrization_is_compactly_recurrent}]
    Let $\lambda \in [0,\frac12]$ and $K \in \Ktt$. We set
    \begin{equation}
        C(S^\lambda,K)\coloneqq \left\{L \in \Ktt \mid \vol L=\vol K \text { and } \diam L \leq 5 \sqrt{\vol K}\right\}.
    \end{equation}
    Since $S^\lambda$ preserves volume, Lemma~\ref{lemma:probability_that_infinitely_many_events_occur_is_1} shows that the random $\lambda$-Steiner process starting at $K$ almost surely visits $C(S^\lambda,K)$ infinitely often. The compactness of $C(S^\lambda,K)$ follows from Lemma~\ref{lemma:characterization_of_compactness}.
\end{proof}

With Proposition~\ref{proposition:planar_lambda_steiner_symmetrization_is_compactly_recurrent} at hand, it is now easy to prove the result about random shaking in the plane that was already mentioned in Section~\ref{section:introduction}.

\begin{corollary} \label{corollary:random_shaking_process_converges_to_any_triangle}
    For $K\in \Ktt$, the random shaking process starting at $K$ converges almost surely modulo $\SL(2,\R)$ to any triangle $T\in \Ktt$ with $\vol T= \vol K$.
\end{corollary}
\begin{proof}
    As shown by Campi--Colesanti--Gronchi \cite{ccg01}, the triangle $T\coloneqq \conv\{0,e_1,e_2\}$ is $S^0$-universal. (That any planar convex body can be brought arbitrarily close to some triangle was already shown by Blaschke \cite{bla17}.) Since $T$ is $S^0$-invariant modulo $\SL(2,\R)$, this shows that any $L \in \Ktt$ is $S^0$-tethered modulo $\SL(2,\R)$. Moreover, by Lemma~\ref{lemma:sigma_invariance_and_sigma_universality}, it follows that triangles are the only $S^0$-invariant bodies in the plane modulo $\SL(2,\R)$. Now the claim follows from Proposition~\ref{proposition:planar_lambda_steiner_symmetrization_is_compactly_recurrent} and Theorem~\ref{theorem:characterization_tethered}, using that shaking preserves volume.
\end{proof}

\begin{remark}
    Numerical experiments suggest that the convergence of the random shaking process starting at $K\in \Ktt$ to the triangle is usually quite fast. It is natural to ask for quantitative bounds on the speed of convergence; corresponding questions for (random) Steiner symmetrization have been studied by Hadwiger \cite{had52}, Bourgain--Lindenstrauss--Milman \cite{blm89}, Bianchi--Gronchi \cite{bg03}, Klartag--Milman \cite{km03} and Klartag \cite{kla04}. Clearly, an analysis of the speed of convergence for random shaking processes would require replacing the abstract framework from Section~\ref{section:random_convergence_of_generalized_steiner_processes} by methods that are more concretely tailored to the structural properties of $\lambda$-Steiner symmetrization.
\end{remark}

\section{A characterization of ellipses and triangles} \label{section:a_characterization_of_ellipses_and_triangles}

In Section~\ref{section:compact_recurrence_of_planar_lambda_steiner_processes}, we have shown that planar random $\lambda$-Steiner processes are compactly recurrent. In view of Theorem~\ref{theorem:characterization_tethered} and Theorem~\ref{theorem:sufficient_condition_errant}, this resolves the planar version of Question~\ref{question:precise_question} for all triples $(K,\lambda,G)$ where $K$ is either $S^\lambda$-errant or $S^\lambda$-tethered modulo $G$. We already know that any $K\in \Ktt$ is $S^{\frac12}$-tethered modulo $\{\Id\}$ and $S^0$-tethered modulo $\SL(2,\R)$. In this section, we show that the cases $\lambda \in \{0,\frac12\}$ are exceptional in the following sense: for every other value of $\lambda \in [0,\frac12]$, any $K \in \Ktt$ is not only not $S^\lambda$-tethered, but even $S^\lambda$-errant modulo any closed $G\leq\GL(2,\R)$, because there is no $S^\lambda$-invariant body modulo $G$. In other words, the property of being $S^\lambda$-invariant for some $\lambda \in [0,\frac12]$ characterizes ellipses and triangles.
\begin{definition}
    We say that $K\in \Knn$ is \emph{$\lambda$-invariant} if it is $S^\lambda$-invariant modulo $\GL(n,\R)$.
\end{definition}
\begin{theorem} \label{theorem:characterization_of_ellipses_and_triangles}
    Let $K \in \Ktt$ and $\lambda \in [0,\frac12]$ such that $K$ is $\lambda$-invariant. Then either $\lambda=0$ and $K$ is a triangle or $\lambda=\frac12$ and $K$ is an ellipse.
\end{theorem}
The remainder of this section is devoted to proving Theorem~\ref{theorem:characterization_of_ellipses_and_triangles}.

According to Definition \ref{definition:lambda_steiner_symmetrization}, $\lambda$-Steiner symmetrization in direction $u\in \sphere$ is a map $\Knn \rightarrow \Knn$. For some purposes, however, it is useful to think of $\lambda$-Steiner symmetrization as acting on the points of $K$ individually: the chords in direction $u$ are moved in their respective affine hulls until the hyperplane $u^\bot$ cuts them at the ratio $\lambda$.%

\begin{definition} \label{definition:pointwise_lambda_symmetrization}
    Let $K \in \Knn$ and $u \in \sphere$. For $\lambda \in [0,\frac12]$, we define the function 
    \begin{equation}
       s^\lambda_{K,u}\colon (K+\linh\{u\})\rightarrow \R, \quad x \mapsto (1-\lambda)f_{K,u}(\pi_{u^\bot}(x))-\lambda g_{K,u}(\pi_{u^\bot}(x))
    \end{equation}
    and define the \emph{pointwise $\lambda$-Steiner symmetrization} as the map
    \begin{equation}
        S_{K,u}^\lambda \colon  K \rightarrow \Sulambda K, \quad x \mapsto x-s_{K,u}^\lambda(x)u.
    \end{equation}
    We say that $u$ is \emph{$\lambda$-flat for $K$} if $S_{K,u}^\lambda$ is affine, which is equivalent to $s_{K,u}^\lambda$ being affine.
\end{definition}

We recover the usual $\lambda$-Steiner symmetrization map $\Knn\rightarrow\Knn$ by $\Sulambda K={S^\lambda_{K,u}(K)}$. %
Defining the \emph{$\lambda$-section} of $K$ in direction $u$ by
\begin{equation} \label{equation:definition_lambda_section}
    H_u^\lambda K \coloneqq(S_{K,u}^\lambda)^{-1}(u^\bot)=\{x+((1-\lambda)f_{K,u}(x)-\lambda g_{K,u}(x))u\mid x\in \pi_{u^\bot}(K)\}
\end{equation}
the direction $u$ is $\lambda$-flat if and only if $H_u^\lambda K$ is a hyperplane section of $K$. %

\begin{lemma} \label{lemma:set_of_lambda_flat_directions_is_closed}
    For $K \in \Knn$ and $\lambda \in [0,\frac12]$, the set of $\lambda$-flat directions for $K$ is a closed set in $\sphere$.
\end{lemma}

\begin{proof}
    We first observe that $f_{K,u}$ is lower semicontinuous because $K$ is closed. Since $f_{K,u}$ is also convex, it follows that the restriction $f_{K,u}|_I$ to an arbitrary line segment $I\subset \pi_{u^\bot}(K)$ is continuous. Because the same reasoning holds for $g_{K,u}$, the function $s_{K,u}^\lambda$ is affine if and only if $s_{K,u}^\lambda|_{\interior K}$ is affine. In the following, we consider $s_{K,u}^\lambda|_{\interior K}$, which has the advantage of being continuous.
    
    Let $\mathcal A$ be the class of affine functions $\interior K\rightarrow \R$. We consider the quantity
    \begin{equation}
          d_\lambda(K,u) \coloneqq  \inf_{a \in \mathcal A}  \int_{\interior K} |s^\lambda_{K,u}(x)-a(x)|^2 \dint x.
    \end{equation}
    Since the $L^2$-norm $\norm{f}_{2} \coloneqq \left(\int_{\interior K} |f(x)|^2 \dint x\right)^\frac12$
    is a norm on the vector space $V$ of bounded continuous functions $\interior K \rightarrow \R$ and since $\mathcal A$ is a closed subset of this normed space, the claim follows if we show that $u\mapsto s^\lambda_{K,u}$ is a continuous map $\sphere \rightarrow (V,\dotnorm_2)$ .
    For this, let $(u_i)_{i \in \N}$ be a sequence in $\sphere$ that converges to $u \in \sphere$. Since $u\mapsto f_{K,u}(\pi_{u^\bot}(x))$ and $u\mapsto g_{K,u}(\pi_{u^\bot}(x))$ are both continuous for every $x \in \interior K$, the sequence of functions $(s^\lambda_{K,u_i})_{i \in \N}$ converges pointwise to $s^\lambda_{K,u}$. Moreover, because $\sup_{x \in \interior K}|s^\lambda_{K,u}(x)| \leq \sup_{x \in \interior K}\norm{x}<\infty$ for all $u \in \sphere$, it follows from the dominated convergence theorem that $(s^\lambda_{K,u_i})_{i \in \N}$ also converges to $s^\lambda_{K,u}$ in $(V,\dotnorm_2)$, as claimed.
\end{proof}

It is immediate from Definition~\ref{definition:pointwise_lambda_symmetrization} that $\Sulambda K$ is affinely equivalent to $K$ for any $\lambda$-flat direction $u\in \sphere$. Since $S^{\lambda}_{-u}K$ is a reflection of $\Sulambda K$, it follows that $K$ is $\lambda$-invariant if for every $u\in \sphere$ at least one of the directions $u$ and $-u$ is $\lambda$-flat. This condition is satisfied by triangles for $\lambda=0$ and by ellipses for $\lambda=\frac12$. Conversely, if we assume that $u$ or $-u$ is $\lambda$-flat for $K \in \Ktt$ and every $u \in \sphereone$, then it is relatively straightforward to show that $K$ must be an ellipse or a triangle (by Lemma~\ref{lemma:set_of_lambda_flat_directions_is_closed}, there exists a $u\in\sphereone$ such that both $u$ and $-u$ are $\lambda$-flat). But this cannot be used to solve our problem. The fact that $\Sulambda K$ is affinely equivalent to $K$ does not imply that $u$ or $-u$ is $\lambda$-flat for $K$, because the affine equivalence can be witnessed by an affine map that is completely unrelated to $S_{K,u}^\lambda$ and $S_{K,-u}^\lambda$. On the other hand, $u$ is trivially $\lambda$-flat for $\Sulambda K$. In order to state what this tells us about $K$, we need the following notion from geometric tomography \cite[Ch.~1,2]{gar06}.

\begin{definition}
    The \emph{(parallel) X-ray of $K \in \Knn$ in direction $u \in \sphere$} is the pair $(S_uK,u)\in\Knn\times \sphere$. We say that two X-rays $(S_uK,u)$ and $(S_{\tilde u}K, \tilde u)$ are \emph{affinely equivalent} if there exists an affine automorphism $\varphi$ of $\R^n$ with
    \begin{equation}\label{eq:induced_bijection_hat}
         S_{\tilde u}K = \varphi(S_uK) \quad \text{and} \quad \tilde u = \hat \varphi(u)\coloneqq  \frac{\varphi(u)-\varphi(0)}{\norm{\varphi(u)-\varphi(0)}}.
    \end{equation}
\end{definition}

The X-ray in direction $u$ can also be defined as the function $|f_{K,u} + g_{K,u}| \colon \pi_{u^\bot}(K) \rightarrow\R$, but the definition in terms of the pair $(S_uK,u)$ has the advantage of providing a natural notion of affine equivalence for two X-rays of $K$. Using this notion, we can now formulate precisely what information we get from the affine equivalence of $\Sulambda K$ and $K$.

\begin{lemma}\label{lemma:lambda_steiner_affine_equivalence_flatness}
    If $\Sulambda K$ is affinely equivalent to $K$, then there exists a $\lambda$-flat direction $v \in \sphere$ such that the X-rays of $K$ in the directions $u$ and $v$ are affinely equivalent.
\end{lemma}

In order to prove Lemma~\ref{lemma:lambda_steiner_affine_equivalence_flatness}, we briefly study how $\lambda$-Steiner symmetrization interacts with affine automorphisms of $\R^n$.

\begin{lemma} \label{lemma:lambda_steiner_affine_transformation}
    Let $\lambda \in [0,\frac12]$,  $K \in \Knn$ and $u \in \sphere$. If $\varphi$ is an affine automorphism of $\R^n$, then the bijective map
    \begin{equation}
        S^\lambda_{\varphi(K),\hat \varphi(u)}\circ \varphi\circ (S^\lambda_{K,u})^{-1}\colon  \Sulambda K \rightarrow S_{\hat \varphi(u)}^\lambda \varphi(K)
    \end{equation}
    extends to an affine automorphism $\psi$ of $\R^n$, which satisfies
    \begin{equation}
        \hat\psi(u)=\hat\varphi(u) \quad \text{and} \quad \psi(\Sulambda K)=S^\lambda_{\hat \varphi(u)}\varphi(K).
    \end{equation}
    If $\varphi$ is linear and angle-preserving, then $\psi=\varphi$. 
    The direction $u$ is $\lambda$-flat for $K$ if and only if $\hat\varphi(u)$ is $\lambda$-flat for $\varphi(K)$.
\end{lemma}

\begin{proof}
    Using the notation from \eqref{eq:induced_bijection_hat}, we write $v\coloneqq\hat\varphi(u)$ and $c \coloneqq \norm{\varphi(u)-\varphi(0)}$. We first observe that
    \begin{equation}
        f_{\varphi(K),v}(\pi_{v^\bot}(\varphi(x)))=cf_{K,u}(\pi_{u^\bot}(x))+\scpr{v,\varphi(\pi_{u^\bot}(x))}
    \end{equation}
    and
    \begin{equation}
        g_{\varphi(K),v}(\pi_{v^\bot}(\varphi(x)))=c g_{K,u}(\pi_{u^\bot}(x))-\scpr{v,\varphi(\pi_{u^\bot}(x))}
    \end{equation}
    for $x \in  (K+\linh\{u\})$, which implies
    \begin{equation}
        s_{\varphi(K),v}^\lambda(\varphi(x)) = c s^\lambda_{K,u}(x)+\scpr{v,\varphi(\pi_{u^\bot}(x))} \quad \text{for all } x\in (K+\linh\{u\}).
    \end{equation}
    For any $x \in  \Sulambda  K$, this yields
    \begin{align*}
        (S^\lambda_{\varphi(K),v}\circ \varphi\circ (S^\lambda_{K,u})^{-1})(x)&=(S^\lambda_{\varphi(K),v}\circ \varphi)(x+s_{K,u}^\lambda(x)u)\\
        &=S^\lambda_{\varphi(K),v}\left(\varphi(x)+s_{K,u}^\lambda(x)(\varphi(u)-\varphi(0))\right)\\
        &=\varphi(x)+s_{K,u}^\lambda(x)(\varphi(u)-\varphi(0))-s_{\varphi(K),v}^\lambda\left(\varphi(x)+s_{K,u}^\lambda(x)(\varphi(u)-\varphi(0))\right)v\\
        &=\varphi(x)+c s_{K,u}^\lambda(x)v-s^\lambda_{\varphi(K),v}\left(\varphi(x)\right)v\\
        &=\varphi(x)-\scpr{v,\varphi(\pi_{u^\bot}(x))}v,
    \end{align*}
    where the last expression is affine in $x$. If $\varphi$ is linear and angle-preserving, then
    \begin{equation}
        \norm{\varphi(u)}\cdot\scpr{v,\varphi(\pi_{u^\bot}(x))}=\scpr{\varphi(u),\varphi(\pi_{u^\bot}(x))}=0.
    \end{equation}
    By definition, we have $\psi \circ S^\lambda_{K,u}=S^\lambda_{\varphi(K),v}\circ \varphi$, and since $\varphi$ and $\psi$ are affine bijections, it follows that $S^\lambda_{K,u}$ is affine if and only if $S^\lambda_{\varphi(K),v}$ is affine. But these statements are clearly equivalent to $s_{K,u}^\lambda$ being affine and  $s_{\varphi(K),v}^\lambda$ being affine, respectively.%
\end{proof}
Lemma~\ref{lemma:lambda_steiner_affine_equivalence_flatness} is now an easy consequence of Lemma~\ref{lemma:lambda_steiner_affine_transformation}.
\begin{proof}[Proof of Lemma~\ref{lemma:lambda_steiner_affine_equivalence_flatness}]
    By assumption, there exists an affine bijection $\varphi$ with $\Sulambda K=\varphi(K)$.
    Using the notation from \eqref{eq:induced_bijection_hat},  Lemma~\ref{lemma:lambda_steiner_affine_transformation} yields the existence of an affine automorphism $\psi$ of $\R^n$ with
    \begin{equation}
        S_u K=S_u \Sulambda K =S_{\hat\varphi(\hat\varphi^{-1}(u))} \varphi(K)=\psi S_{\hat\varphi^{-1}(u)}K \quad \text{and} \quad \hat\psi(\hat\varphi^{-1}(u))=\hat\varphi(\hat\varphi^{-1}(u))=u.
    \end{equation}
    Since $u$ is $\lambda$-flat for $\Sulambda K=\varphi(K)$, Lemma~\ref{lemma:lambda_steiner_affine_transformation} further implies that $\hat\varphi^{-1}(u)$ is $\lambda$-flat for $K$.
\end{proof}

Lemma~\ref{lemma:lambda_steiner_affine_equivalence_flatness} tells us precisely what we know about a convex body $K\in \Knn$ that is $\lambda$-invariant for a given $\lambda\in[0,\frac12]$. In the remainder of this section, we will use this information to show that in the planar case, no such body exists for $\lambda \in (0,\frac12)$. Our intermediate goal is to prove the following statement.%

\begin{proposition} \label{proposition:there_exists_a_circular_arc_of_flat_directions}
    Let $\lambda \in (0,\frac12]$ and let $K \in \Ktt$ be $\lambda$-invariant. Then there exists a non-empty, open circular arc $I \subset \sphereone$ such that every $v \in I$ is $\lambda$-flat for $K$.
\end{proposition}

In order to prove Proposition~\ref{proposition:there_exists_a_circular_arc_of_flat_directions}, we study the \emph{affine symmetry group} of a $S^\lambda$-invariant body $K$, which consists of all affine automorphisms $\varphi \colon \R^n \rightarrow \R^n$ such that $\varphi(K)=K$ and whose group operation is given by function composition. It will be a bit more convenient to work with the Euclidean \emph{(full) symmetry group} $G(K)$ of $K$, which contains exactly the affine symmetries $\varphi$ that are isometries. If we assume that $K\in \Knn$ is in John position, i.e., that the unique ellipsoid of maximal volume inside $K$ is the Euclidean unit ball, then the affine symmetry group of $K$ coincides with $G(K)$ and $G(K)$ is a subgroup of $\Orthogonal(n)$. This follows from the fact that the John position is unique up to linear isometries. The uniqueness up to isometries can also be expressed by saying that there exists a unique symmetric positive definite matrix $A\in \R^{n \times n}$ and a unique $z \in \R^n$ such that $AK +z$ is in John position. We write $J(K)\coloneqq AK+z$ for this canonical choice of the John position and note that $K \mapsto J(K)$ is continuous and satisfies $J(OK)=O(JK)$ for $O\in \Orthogonal(n)$.

\begin{lemma} \label{lemma:invariant_body_has_a_reflection_symmetry}
    Let $\lambda \in [0,\frac12]$ and let $K \in \Ktt$ be $\lambda$-invariant. Then the affine symmetry group of $K$ contains an affine automorphism $\varphi$ with $\det \varphi=-1$. Equivalently, $K$ has a $\frac12$-flat direction.
\end{lemma}

\begin{proof}
    By Lemma~\ref{lemma:lambda_steiner_affine_transformation}, the property of $\lambda$-invariance is affinely invariant. Since the claim of the lemma is also affinely invariant, we can assume without loss of generality that $K$ is in John position.
    
    We consider the map $\sphereone \rightarrow \Ktt$ given by $u \mapsto K_u \coloneqq J(\Sulambda K)$ and set
    \begin{equation}
        S^+ \coloneqq \{u \in \sphereone \mid \exists O_u \in \SO(2): K_u=O_u(K) \}
    \end{equation}
    and
    \begin{equation}
        S^- \coloneqq \{u \in \sphereone \mid \exists O_u \in \Orthogonal(2) \setminus \SO(2): K_u=O_u(K) \}.
    \end{equation}
    Since $K$ is $\lambda$-invariant, for every $u \in \sphereone$ there exists an affine bijection $\varphi$  with $K_u =\varphi(K)$. But since both $K$ and $K_u$ are in John position, we have $\varphi \in \Orthogonal(2)$. This shows that $\sphereone=S^+ \cup S^-$. We claim that both $S^+$ and $S^-$ are closed. Let $(u_i)_{i \in \N}$ be a convergent sequence in $S^+$, converging to $u\in \sphereone$. By definition, there exists a sequence $(O_i)_{i \in \N}$ of matrices in $\SO(2)$ with $K_{u_i}=O_i(K)$ for $i \in \N$. Since $\SO(2)$ is compact, $(O_i)_{i \in \N}$ contains a convergent subsequence, whose limit $O \in \SO(2)$ satisfies $K_u=O(K)$. The claim that $S^-$ is closed is shown similarly. Let $u \in \sphereone$ be arbitrary. Since $K_u$ and $K_{-u}$ are mirror images across the hyperplane $u^\bot$, both $S^+$ and $S^-$ are non-empty. Because $\sphereone$ is connected, it follows that there exists a direction $u \in S^+ \cap S^-$. Hence, there exist $O^+ \in \SO(2)$ and $O^- \in \Orthogonal(2) \setminus \SO(2)$ with $O^+(K)=K_u = O^-(K)$. It follows that the reflection $(O^+)^{-1}O^-$ is a symmetry of $K$ and that both directions orthogonal to its axis of reflection are $\frac12$-flat for $K$.
\end{proof}

We already know that a $\lambda$-invariant body $K$ has at least one $\lambda$-flat direction; the following lemma implies that $K$ has at least two distinct $\lambda$-flat directions if $\lambda\in(0,\frac12)$.

\begin{lemma} \label{lemma:there_exist_two_non_equivalent_x_rays}
    Let $K\in\Ktt$ be $\lambda$-invariant for $\lambda \in (0,\frac12)$. Then $K$ has two X-rays that are not affinely equivalent.
\end{lemma}

\begin{proof}
    To show the contrapositive, we assume that any two X-rays of $K$ are affinely equivalent. Because $K$ is $\lambda$-invariant, there exists a $\lambda$-flat direction $u$. We also know from Lemma~\ref{lemma:invariant_body_has_a_reflection_symmetry} that there exists a $\frac12$-flat direction $\tilde u$. We recall the definition of the $\lambda$-section from \eqref{equation:definition_lambda_section}. Let $v$ and $\tilde v$ be two directions which are parallel to the line segments $H_u^\lambda K$ and $H_{\tilde u}^{1/2} K $, respectively.  %
    There are exactly two affine bijections $[0,1] \rightarrow \pi_{v^\bot}(K)$; let $\gamma$ be the one that satisfies $\gamma(\lambda)\in \pi_{v^\bot}(H_u^\lambda K)$. We consider the function
    \begin{equation}
        h\coloneqq|f_{K,v}+g_{K,v}| \circ \gamma \colon [0,1] \rightarrow \R.
    \end{equation}
    Because $u$ is $\lambda$-flat and $v$ is parallel to $H_u^\lambda K$, $h$ exhibits the symmetry $h(x)=h\left(1-\frac{1-\lambda}{\lambda}\cdot x\right)$ for all $x\in [0,\lambda]$. Choosing $\tilde \gamma$ to be any of the two affine bijections $[0,1] \rightarrow \pi_{\tilde v^\bot}(K)$ and setting $\tilde h\coloneqq|f_{K,\tilde v}+g_{K,\tilde v}| \circ \tilde\gamma$, we similarly get $\tilde h(x)=\tilde h\left(1- x\right)$ for all $x\in [0,1]$. Now the assumption that any two X-rays are affinely equivalent implies that $h=c\tilde h$ for some $c>0$, and we conclude that $ h(x)= h\left(1- x\right)$ for all $x\in [0,1]$.
    Because $h$ is concave, this implies that $h$ attains global maxima at $\frac12$ and $\lambda$. Let $I \coloneqq h^{-1}(h(\lambda))$ be the interval of all global maximizers and set $a \coloneqq \inf I \leq \lambda$ and $b \coloneqq \sup I\geq \frac12$. Due to the symmetries of $h$, we have
    \begin{equation}
        a=1-b \quad \text{and} \quad a= \frac{\lambda}{1-\lambda}(1-b).
    \end{equation}
    This implies $\lambda=\frac12$ or $b=1$ (and hence $a=0$). Since $\lambda \neq \frac12$, we get that $h$ is constant, which implies that $K$ is a parallelogram. But a parallelogram clearly does not have the property that any two X-rays are affinely equivalent, so we conclude that $\lambda=\frac12$.
\end{proof}

On the other hand, if two distinct $\lambda$-flat directions of $K \in \Knn$ have affinely equivalent X-rays, we obtain an affine symmetry of $K$.

\begin{lemma} \label{lemma:two_lambda_flat_directions_affine_symmetry}
    Let $K \in \Knn$ and $\lambda \in [0,\frac12]$. If the directions $u,v\in\sphere$ are $\lambda$-flat for $K$ and the X-rays of $K$ in the directions $u$ and $v$ are affinely equivalent, then there exists an affine symmetry $\psi$ of $K$ with $\hat{\psi}(u)=v$.
\end{lemma}

\begin{proof}
    By assumption, there exists an affine automorphism $\varphi$ of $\R^n$ with $\varphi(S_uK)=S_vK$ and $\hat \varphi(u)=v$. By Lemma~\ref{lemma:lambda_steiner_affine_transformation}, the map
    \begin{equation}
        S^\lambda_{\varphi(S_uK),v}\circ \varphi\circ (S^\lambda_{S_uK,u})^{-1}\colon  \Sulambda S_u K \rightarrow S_v^\lambda \varphi(S_uK)
    \end{equation}
    extends to an affine automorphism $\psi_0$ of $\R^n$, which satisfies $\hat \psi_0(u)=\hat \varphi(u)=v$ and $\psi_0(\Sulambda S_uK)=S^\lambda_v\varphi(S_uK)$. This yields
    \begin{equation}
        \psi_0(\Sulambda K)=\psi_0(\Sulambda S_uK)=S_v^\lambda(\varphi(S_uK))=S_v^\lambda S_vK=S_v^\lambda K.
    \end{equation}
    Since $u$ and $v$ are $\lambda$-flat for $K$, $\psi\coloneqq (S_{K,v}^\lambda)^{-1} \circ \psi_0\circ S_{K,u}^\lambda$ is affine and has the desired properties.
\end{proof}

We also need the following elementary lemma, whose proof is included for the sake of completeness.

\begin{lemma} \label{lemma_infinite_symmetry_group_ball}
    Let $K \in \Ktt$. If the symmetry group $G(K)$ of $K$ is infinite, then $K$ is a disk.
\end{lemma}
\begin{proof}
    Without loss of generality, we assume that $K$ is centered, which ensures that $G(K)\leq\Orthogonal(2)$. If $G(K) \setminus \SO(2)$ is non-empty, then it has the same cardinality as $G(K) \cap \SO(2)$, so $G(K) \cap \SO(2)$ will always be infinite. Because $\SO(2)$ is compact, the subgroup $G(K) \cap \SO(2)$ cannot be discrete and hence contains rotations by arbitrarily small positive angles. Because $G(K) \cap \SO(2)$ is closed, we have $G(K) \supset \SO(2)$, which implies the claim.
\end{proof}

We are now ready to prove Proposition~\ref{proposition:there_exists_a_circular_arc_of_flat_directions}. 
\begin{proof}[Proof of Proposition~\ref{proposition:there_exists_a_circular_arc_of_flat_directions}]
    By Lemma~\ref{lemma:lambda_steiner_affine_transformation}, we can assume without loss of generality that $K$ is in John position.    If $G(K)$ is infinite, then  Lemma~\ref{lemma_infinite_symmetry_group_ball} implies that $K$ is a disk. Since the disk is only $\lambda$-invariant for $\lambda=\frac12$, we have $\lambda=\frac12$ and the claim is trivial. On the other hand, if $\lambda=\frac12$, then Lemma~\ref{lemma:sigma_invariance_and_sigma_universality} implies that $K$ is an ellipse. Because $K$ is in John position, it is a disk and the claim is trivial.
    
    It remains to consider the case where $G(K)$ is finite and $\lambda \in (0,\frac12)$. As a finite subgroup of $\Orthogonal(2)$, $G(K)$ has the structure of a dihedral or cyclic group of order $m \geq 1$. By Lemma~\ref{lemma:invariant_body_has_a_reflection_symmetry}, we know that $G(K)$ contains at least one reflection. So $G(K)$ is actually a dihedral group of order $m \geq 2$.

    The group action of $G(K)$ partitions the set $\sphereone$ into orbits, i.e., equivalence classes of the form
    \begin{equation}
        [u] \coloneqq \{Ou \mid O \in G(K)\}.
    \end{equation}
    A \emph{fundamental domain} is a subset of $\sphereone$ that contains exactly one representative of every orbit. In the case of the dihedral group of order $m\geq 2$, a fundamental domain is given by any inclusion-wise maximal circular arc $F \subset \sphereone$ whose interior does not contain fixed points of reflections in the group. Such a circular arc is closed, has arc length $\frac{2\pi}{m}$, and its boundary points are fixed points of reflections. Let $F$ be such a fundamental domain of $G(K)$.

    Since $K$ is $\lambda$-invariant, Lemma~\ref{lemma:lambda_steiner_affine_equivalence_flatness} asserts that for every $u \in F$ there exists a $\lambda$-flat direction $\phi(u) \in F$ whose X-ray is affinely equivalent to the X-ray in direction $u$. %
    We claim that any map $\phi \colon F\rightarrow F$ defined in this way must be continuous. For this, we consider a sequence $(v_i)_{i \in \N}$ in $F$ that converges to $v$. Let $w \in F$ be an accumulation point of $\{\phi(v_i)\mid i \in \N\}$. By Lemma~\ref{lemma:set_of_lambda_flat_directions_is_closed}, $w$ is $\lambda$-flat. Moreover, we claim that the X-ray of $K$ in direction $w$ is affinely equivalent to that in direction $\phi(v)$. Up to choosing a subsequence, we can assume that $\phi(v_i)\goestoinfty{i} w$. By the definition of $\phi$, there exist two sequences $A_1,A_2,\dots\in \GL(2,\R)$ and $z_1,z_2,\dots \in \R^2$ with $A_iS_{v_{i}}K+z_i=S_{\phi(v_{i})}K$ and $\phi(v_i)=\frac{A_iv_i}{\norm{A_iv_i}}$ for $i\in\N$. Since $(v_{i})_{i \in \N}$ and $(\phi(v_{i}))_{i \in \N}$ are convergent and Steiner symmetrization is continuous, the norms $\norm{A_i}$, $\norm{A_i^{-1}}$ and $\norm{z_i}$ are uniformly bounded in $i \in \N$, so we can find a subsequence with $A_{k(i)}\goestoinfty{i}A^*$ and $z_{k(i)}\goestoinfty{i}z^*$. Then we have $A^*S_vK+z^*=S_wK$ and $w=\frac{A^*v}{\norm{A^*v}}$, i.e., the X-rays of $K$ in the directions $v$ and $w$ are affinely equivalent. Moreover, by the definition of $\phi$, the X-ray in direction $\phi(v)$ is equivalent to that in direction $v$ and hence to that in direction $w$. By Lemma~\ref{lemma:two_lambda_flat_directions_affine_symmetry}, there exists an affine symmetry $\psi$ of $K$ that maps $\phi(v)$ to $w$.
    Since $K$ is in John position, we have $\psi \in G(K)$. This shows that $\phi(v),w \in F$ are in the same orbit and because $F$ is a fundamental domain, we have $\phi(v)=w$. Since $w$ was an arbitrary accumulation point of $\{\phi(v_i)\mid i \in \N\}$, it follows that $\phi$ is continuous.

    We have shown that there exists a continuous map $\phi \colon F \rightarrow F$ whose image contains only $\lambda$-flat directions. Since $\phi$ is continuous, $\phi(F)$ is a closed circular arc. Using that $\lambda \in (0,\frac12)$, Lemma~\ref{lemma:there_exist_two_non_equivalent_x_rays} implies that $\phi(F)$ has cardinality at least two. This shows that the relative interior of $\phi(F)$ is an open circular arc with the desired properties.
\end{proof}

At this point, we still have basically no information about the boundary structure of a $\lambda$-invariant body $K \in \Ktt$, and there is a priori no reason why it might not be highly irregular. In the range $\lambda\in(0,\frac12]$, the utility of Proposition~\ref{proposition:there_exists_a_circular_arc_of_flat_directions} is that it forces the boundary of $K$ to be highly regular, namely $C^2_+$. This will allow us to analyze the curvature of $K$ using a second-order differential equation. Before we proceed, we briefly recall the relevant basic notions from convex geometry (see \cite[Ch.~2]{sch13}).

\begin{definition}
    Let $K \in \Ktt$ and let $x \in \bd K$ be a \emph{regular} point of $K$, i.e., there exists a unique outer unit normal vector of $K$ at $x$, which we denote by $n(x)$. The regular boundary point $x$ is called \emph{normal} if $g_{K,n(x)}$ is twice differentiable at $\pi_{n(x)^\bot}(x)$ in the sense that there exists a second-order Taylor expansion. In this case, the (principal) curvature of $K$ at $x$ is given by
    \begin{equation} \label{equation:formula_for_kappa}
        \kappa(K,x) = g_{K,n(x)}''(\pi_{n(x)^\bot}(x)) \geq 0,
    \end{equation}
    where $g_{K,n(x)}''$ is to be interpreted as in Remark~\ref{remark:meaning_of_g''} below.
\end{definition}

\begin{remark} \label{remark:meaning_of_g''}
    Let $u \in \sphereone$. In order to compute first and second derivatives of $f_{K,u}$ and $g_{K,u}$, we choose a vector $v \in \sphereone \cap u^\bot$, say, the one which results from $u$ by a 90° clockwise rotation, and identify $f_{K,u}$ and $g_{K,u}$ with the functions given by
    \begin{equation}
        t \mapsto f_{K,u}(t\cdot v) \quad \text{and} \quad t \mapsto g_{K,u}(t\cdot v)
    \end{equation}
    on suitable intervals. This allows us to identify $f'_{K,u},g'_{K,u},f''_{K,u},g''_{K,u}$ with scalars. Because the calculations below do not depend on the chosen orientation, the choice of $v \in \sphereone \cap u^\bot$ has only been suggested for the sake of concreteness, but is irrelevant for our purposes.
\end{remark}

Although the boundary of an arbitrary convex body can be highly irregular, the set of boundary points that are not normal always has $(n-1)$-dimensional Hausdorff measure zero. In dimension $n\leq 3$, this was shown by Busemann--Feller \cite{bf36}; the general case is due to Aleksandrov \cite{ale39}. 

For $\lambda$-invariant bodies in the setting $\lambda\in(0,\frac12]$, Proposition~\ref{proposition:there_exists_a_circular_arc_of_flat_directions} allows us to transfer this knowledge about regularity to all boundary points of $K$, yielding the following result.

\begin{lemma} \label{lemma:boundary_is_C2}
    Let $K \in \Ktt$ be $\lambda$-invariant for $\lambda \in (0,\frac12]$. %
    Then $K$ has a $C^2_+$ boundary, i.e., $\bd K$ is $C^2$ and we have $\kappa(K,x)>0$ for all $x \in \bd K$.
\end{lemma}

In order to prove Lemma~\ref{lemma:boundary_is_C2}, we need the following elementary result from calculus, whose proof we omit.

\begin{lemma} \label{lemma:curvature_formula}
    Let $K \in \Ktt$, $x \in \bd K$ be regular and $u \in \sphereone$ with $\scpr{u,n(K,x)}>0$. Then $g_{K,u}$ is twice differentiable at $\pi_{u^\bot}(x)$ if and only if $x$ is a normal boundary point of $K$. In this case, we have
    \begin{equation}
        g_{K,u}''(\pi_{u^\bot}(x))= \kappa(K,x) (1+ (g_{K,u}'(\pi_{u^\bot}(x)))^2)^{3/2}.
    \end{equation}
\end{lemma}

\begin{proof}[Proof of Lemma~\ref{lemma:boundary_is_C2}]
    Let $I$ be the circular arc of $\lambda$-flat directions from Proposition~\ref{proposition:there_exists_a_circular_arc_of_flat_directions}. We first show that every $x \in \bd K$ is regular. For this, we observe that the set
    \begin{equation} \label{equation:definition_of_S}
        S\coloneqq  \{x \in \bd K\mid \nexists v \in I\colon \pi_{v^\bot}(x) \in \relint (\pi_{v^\bot}(K))\}
    \end{equation}
    contains at most two points, and that every $x \in S$ is non-regular. Let $x \in (\bd K) \setminus S$. Because the set of normal boundary points of $K$ is dense in $\bd K$, there exists a direction $v \in I$ such that $\pi_{v^\bot}(x) \in \relint (\pi_{v^\bot}(K))$ and such that the unique boundary point
    \begin{equation} \label{equation:definition_of_y}
        y \in ((x + \linh\{v\}) \cap \bd K)\setminus \{x\} 
    \end{equation}
    is normal and, in particular, regular.   Because $v\in I$ is $\lambda$-flat, the function $-\lambda g_{K,v}+(1-\lambda)f_{K,v}$ is affine and hence either both $f_{K,v}$ and $g_{K,v}$ are differentiable at $\pi_{v^\bot}(x)=\pi_{v^\bot}(y)$ or none of them. Since $y$ is regular, this shows that $x$ is regular as well. Since $x \in (\bd K \setminus S)$ was arbitrary, this shows that the non-regular boundary points of $K$ are precisely the points in $S$, of which there are at most two. We assume towards a contradiction that there exists a point $x \in S$. If $u\in \sphereone$ satisfies $\pi_{u^\bot}(x)\in \relint \pi_{u^\bot}(K)$, then $f_{K,u}+g_{K,u}$ is not  differentiable at $\pi_{u^\bot}(x)$ and the two points in $(x+\linh\{u\}) \cap \bd \Sulambda K$ are non-regular points of $\Sulambda K$. By the same argument, $S_{\tilde u}^\lambda\Sulambda K$ has at least four non-regular boundary points if $\tilde u \in \sphereone \setminus \{u\}$ is sufficiently close to $u$. Since $S_{\tilde u}^\lambda\Sulambda K$ is affinely equivalent to $K$ and the number of non-regular boundary points is an affine invariant, this contradicts the fact that $K$ has at most two non-regular boundary points. It follows that $S=\emptyset$ and that every boundary point of $K$ is regular.
    
    Using a similar argument, we now show that every $x\in \bd K$ is normal. Let $x \in \bd K$. By the same argument as above and using that the set $S$ from \eqref{equation:definition_of_S} is empty, we know that there exists a direction $v\in I$ such that the unique boundary point defined as in \eqref{equation:definition_of_y} is normal. Because $\pi_{v^\bot}(y) \in \relint (\pi_{v^\bot}(K))$, the outer normal vector $n(y)$ is not orthogonal to $v$. We first consider the case where $\scpr{n(y),v}>0$. Let $\tilde I \subset I$ be an open circular arc with $\scpr{n(y),w}>0$ for all $w \in \tilde I$. By Lemma~\ref{lemma:curvature_formula}, we have
    \begin{equation}
        \kappa(K,y)=\frac{g''_{K,w}(\pi_{w^\bot}(y))}{(1+g'_{K,w}(\pi_{w^\bot}(y))^2)^{3/2}} \quad \text{for all } w\in \tilde I.
    \end{equation}
    Moreover, writing $\theta(v,w)$ for the signed angle between $v$ and $w$, we have
    \begin{equation}
        g'_{K,w}(\pi_{w^\bot}(y))=\tan\left(\arctan(g'_{K,v}(\pi_{v^\bot}(y)))+\theta(v,w)\right)\quad \text{for all } w\in \tilde I.
    \end{equation}
    We define a map $z \colon \tilde I \rightarrow \bd K$ by considering the unique boundary point
    \begin{equation}
        z(w) \in ((y + \linh\{w\}) \cap \bd K)\setminus \{y\} 
    \end{equation}
    for every $w \in \tilde I$.
    Since every $w \in \tilde I$ is $\lambda$-flat, we know that the function $(1-\lambda)f_{K,w}-\lambda g_{K,w}$ is affine and hence twice differentiable for all $w \in \tilde I$. This implies that every $z(w)$, $w\in \tilde I$, is normal with
    \begin{align} \label{eq:lemma:boundary_is_C2:formula_g''}
        g''_{K,-w}(\pi_{w^\bot}(z(w)))&=f''_{K,w}(\pi_{w^\bot}(z(w)))=\frac{\lambda}{1-\lambda}g''_{K,w}(\pi_{w^\bot}(y))\\
        &=\frac{\lambda}{1-\lambda}\cdot\kappa(K,y)\cdot\left(1+\left(\tan\left(\arctan(g'_{K,v}(\pi_{v^\bot}(y)))+\theta(v,w)\right)\right)^2\right)^{3/2}.
    \end{align}
    Because the right-hand side of \eqref{eq:lemma:boundary_is_C2:formula_g''} is continuous, it follows that $w\mapsto g''_{K,-w}(\pi_{w^\bot}(z(w)))$ is continuous. Moreover, since every boundary point of $K$ is regular, $\bd K$ is a $C^1$ submanifold of $\R^2$ \cite[Thm.~2.2.4]{sch13}, and the function $w\mapsto g'_{K,-w}(\pi_{w^\bot}(z(w)))$ is continuous as well. Therefore, Lemma~\ref{lemma:curvature_formula} implies that $w \mapsto \kappa(K,z(w))$ is continuous.

    In the case where $\scpr{n(y),v}<0$, we can repeat the argument above, reversing the roles of $f_{K,w}=g_{K,-w}$ and $g_{K,w}=f_{K,-w}$. Since $x \in \bd K$ was arbitrary, this concludes the proof that $K$ has a $C^2$ boundary.

    It remains to show that $\kappa(K,\cdot)$ is strictly positive. Because $x \mapsto \kappa(K,x)$ is continuous, the set $N\coloneqq\{x \in \bd K\mid \kappa(K,x)=0\}$ is a relatively closed subset of $\bd K$. Let $x \in N$. Because the set $S$ from \eqref{equation:definition_of_S} is empty, there exists a $v\in I$ such that $\pi_{v^\bot}(x)\in \relint \pi_{v^\bot}(K)$. Let $y \neq x$ be as in \eqref{equation:definition_of_y}. In view of \eqref{eq:lemma:boundary_is_C2:formula_g''} and $x \in N$, we have $y \in N$. Let $\tilde I \subset I$ be open with $v \in \tilde I$ and the property that $\pi_{\tilde v^\bot}(y)\in \relint \pi_{\tilde v^\bot}(K)$ for all $\tilde v\in \tilde I$. Again, by \eqref{eq:lemma:boundary_is_C2:formula_g''}, for every $z \in \bd K$ with $\pi_{\tilde v^\bot}(z)=\pi_{\tilde v^\bot}(y)$ we have $z \in N$. Since $\tilde I$ is open and contains $v$, the set of these points $z \in N$ contains a (relatively) open neighborhood of the arbitrary point $x \in N$, which shows that $N$ is relatively open. Because $N$ is also relatively closed and $\bd K$ is connected, it follows that $N=\bd K$ or $N=\emptyset$. But $N=\bd K$ is impossible because $\bd K$ is not contained in a line.
\end{proof}

Having established that the boundary of a $\lambda$-invariant body $K \in \Ktt$ for $\lambda \in(0,\frac12]$ is $C^2_+$, we can now use a second-order ordinary differential equation to analyze the structure of the boundary. It turns out that the existence of an open circular arc $I \subset \sphereone$ of $\lambda$-flat directions forces $K$ to be an ellipse; this leads to a contradiction unless $\lambda=\frac12$.

\begin{proposition} \label{proposition:if_there_exists_an_arc_then_lambda_is_1/2}
    Let $K \in \Ktt$ be $\lambda$-invariant for $\lambda \in (0,\frac12]$. %
    Then $\lambda =\frac12$.
\end{proposition}
\begin{proof}
    We define $u \colon \R\rightarrow \sphereone$, $u(\theta)\coloneqq (\sin(\theta), \cos(\theta))^\transp$. Without loss of generality, let $e_2=u(0) \in I$, where $I \subset \sphereone$ is the open circular arc of $\lambda$-flat directions from Proposition~\ref{proposition:there_exists_a_circular_arc_of_flat_directions}. We first apply some preparatory modifications to $K$. Because $e_2$ is $\lambda$-flat, the map $S_{K,e_2}^\lambda$ from Definition \ref{definition:pointwise_lambda_symmetrization} is affine (in fact, an affine shear), and Lemma~\ref{lemma:lambda_steiner_affine_transformation} implies that there exists an $\varepsilon>0$ such that every  $u(\theta)$ with $\theta \in (-\varepsilon,\varepsilon)$ is $\lambda$-flat for the body $S^\lambda_{e_2}(K)$. By Lemma~\ref{lemma:boundary_is_C2}, the boundary of $S^\lambda_{e_2}(K)$ is $C^2_+$, and hence
    the concave function $|f_{K,e_2}+g_{K,e_2}|=|f_{S^\lambda_{e_2}(K),e_2}+g_{S^\lambda_{e_2}(K),e_2}|$ has a unique global maximizer $y$. Writing $w(S^\lambda_{e_2}(K),e_2)$ for the width of $S^\lambda_{e_2}(K)$ in direction $e_2$, the global maximum satisfies
    \begin{equation}
        \omega \coloneqq|f_{S^\lambda_{e_2}(K),e_2}(y)+g_{S^\lambda_{e_2}(K),e_2}(y)|=w(S^\lambda_{e_2}(K),e_2).
    \end{equation}
    We consider the unique boundary point $z \in (y + \pos\{-e_2\}) \cap \bd S^\lambda_{e_2}(K)$. By Lemma~\ref{lemma:boundary_is_C2}, $\kappa_z \coloneqq \kappa(S^\lambda_{e_2}(K),z)$ is positive. Setting
    \begin{equation}
        M\coloneq\begin{bmatrix} \sqrt{\frac{\kappa_z\omega}2\cdot\frac{1-\lambda}\lambda} &0 \\
            0& 1
        \end{bmatrix} \quad \text{and} \quad \tilde K \coloneqq M(S^\lambda_{e_2}(K)-z),
    \end{equation}
    we observe that $\tilde K$ has the following properties:
    \begin{enumerate}[label=(\roman*)]
        \item \label{item:property1} The origin is the unique global maximizer of $|f_{\tilde K,e_2}+g_{\tilde K,e_2}|$ and satisfies
        \begin{equation} \label{eq:property1_width}
            |f_{\tilde K,e_2}(0)+g_{\tilde K,e_2}(0)|=-g_{\tilde K,e_2}(0)=w(\tilde K,e_2)=w(S^\lambda_{e_2}(K),e_2)=\omega
        \end{equation}
        as well as
        \begin{equation} \label{eq:formula_for_kappa_at_origin}
            \kappa(\tilde K, 0)=\kappa_z \cdot \frac{2}{\kappa_z\omega} \cdot \frac\lambda{1-\lambda}= \frac{2}{w(\tilde K, e_2)} \cdot \frac\lambda{1-\lambda}.%
        \end{equation}
        Moreover, because $(1-\lambda)\cdot f_{\tilde K,e_2}  -\lambda \cdot g_{\tilde K,e_2}$ is constant, we have
        \begin{equation} \label{eq:property1_derivatives}
            f'_{\tilde K,e_2}(0) = \frac{\lambda}{1-\lambda} g'_{\tilde K,e_2}(0) = 0.
        \end{equation}
        \item \label{item:property2} By Lemma~\ref{lemma:lambda_steiner_affine_transformation}, there exists an $\varepsilon>0$ such that any  $u(\theta)$ with $\theta \in (-\varepsilon,\varepsilon)$ is $\lambda$-flat for the body $\tilde K$. Choosing $\varepsilon\in\left(0,\frac\pi2\right)$ small enough, we can ensure the existence of a $\delta>0$ such that the lines
        \begin{equation}
            \linh \{u(\theta)\}+\delta \cdot u\left(\theta+\frac\pi2\right) \quad\text{and}\quad \linh \{u(\theta)\}- \delta\cdot u\left(\theta+\frac\pi2\right)
        \end{equation}
        intersect the body $\tilde K$ for all $\theta \in (-\varepsilon,\varepsilon)$ and such that the unique boundary point
        \begin{equation} \label{eq:z_theta}
            z(\theta) \in (\linh\{u(\theta)\} \cap \bd \tilde K)\setminus \{0\} 
        \end{equation}
        satisfies
        \begin{equation}
            \norm{\pi_{e_2^\bot}(z(\theta))} < \delta \quad \text{for all } \theta \in (-\varepsilon,\varepsilon).
        \end{equation}
    \end{enumerate}
    By property \ref{item:property2}, the function
    \begin{align}
        h\colon (-\varepsilon,\varepsilon) \times (-\delta,\delta) \rightarrow \R, \quad h(\theta,x)\coloneqq-g_{\tilde K,u(\theta)}\left(x\cdot u\left(\theta+\frac\pi2\right)\right),
    \end{align}     
    is well-defined. Our intermediate goal is to show that $h(\theta,0)=h(0,0)\cos(\theta)=\omega \cos(\theta)$ for $\theta \in (-\varepsilon,\varepsilon)$, which implies that $\bd \tilde K$ contains an arc of a circle.

    We know from Lemma~\ref{lemma:boundary_is_C2} that $h$ is twice continuously differentiable in the second argument for all $\theta \in (-\varepsilon,\varepsilon)$. To show that $h(\cdotbox,0)$ is of class $C^2$, we recall the definition of $z(\theta)$ from \eqref{eq:z_theta} and define
    \begin{equation}
        x \colon (-\varepsilon,\varepsilon)\rightarrow (-\delta,\delta), \quad x(\theta) = \scpr{\pi_{e_2^\bot}(z(\theta)),e_1},
    \end{equation}
    which is again well-defined by property \ref{item:property2}. We observe that
    \begin{align}
        h(\theta,0) = \norm{z(\theta)}_2 = \sqrt{x(\theta)^2+h(0,x(\theta))^2}.
    \end{align}
    Since $\norm{z(\theta)}_2^2 =x(\theta)^2+h(0,x(\theta))^2>0$ for $\theta \in (-\varepsilon,\varepsilon)$ by property \ref{item:property2},  the twofold continuous differentiability of $h(\cdotbox,0)$ is shown if $\theta \mapsto x(\theta)$ is of class $C^2$. Decreasing $\varepsilon>0$ if necessary, the latter follows from the implicit function theorem, in combination with the fact that $x(\theta)$ can be implicitly defined by
    \begin{align}
        F(\theta,x(\theta))\coloneqq \tan (\theta) - \frac{x(\theta)}{h(0,x(\theta))} =0 \quad \text{for } \theta \in (-\varepsilon,\varepsilon),
    \end{align}
    where $F$ is $C^2$ and satisfies $F^{(0,1)}(0,0)=-\frac1{h(0,0)}\neq 0$.%

    The crucial ingredient for the rest of the proof is the following functional equality, which follows immediately from the definition of $h$: We have
    \begin{align}
        h(\theta, \sin(\alpha) \cdot h(\alpha+\theta,0)) = \cos(\alpha) \cdot h(\alpha+\theta,0)
    \end{align}
    for all $\alpha, \theta \in \R$ for which the expression makes sense. Differentiating once with respect to $\alpha$, we get
    \begin{align} \label{eq_first_derivative_of_g}
        h^{(0,1)}(\theta, \sin(\alpha) \cdot h(\alpha+\theta,0)) \left(\cos(\alpha) \cdot h(\alpha+\theta,0)+\sin(\alpha) \cdot h^{(1,0)}(\alpha+\theta,0) \right)\\
        \quad\quad=-\sin(\alpha)\cdot h(\alpha+\theta,0)+\cos(\alpha) \cdot h^{(1,0)}(\alpha+\theta,0).
    \end{align}
    Setting $\alpha=0$, this yields
    \begin{equation}  \label{eq_first_derivative_of_g_alpha=0}
        h^{(0,1)}(\theta,0) \cdot h(\theta,0) = h^{(1,0)}(\theta,0).
    \end{equation}
    If we differentiate \eqref{eq_first_derivative_of_g} once again with respect to $\alpha$, we obtain
    \begin{align}
        &h^{(0,2)}(\theta, \sin(\alpha) \cdot h(\alpha+\theta,0))\left(\cos(\alpha) \cdot h(\alpha+\theta,0)+\sin(\alpha) \cdot h^{(1,0)}(\alpha+\theta,0)\right)^2\\
        &\quad\quad\quad +h^{(0,1)}(\theta, \sin(\alpha)  h(\alpha+\theta,0)) \left(-\sin(\alpha) h(\alpha+\theta,0)+2\cos(\alpha)h^{(1,0)}(\alpha+\theta,0)\right.\\&\hspace{9cm}\left.+\sin(\alpha) \cdot h^{(2,0)}(\alpha+\theta,0)\right)\\
        &\quad\quad =-\cos(\alpha)\cdot h(\alpha+\theta,0)-2\sin(\alpha)\cdot h^{(1,0)}(\alpha+\theta,0) + \cos(\alpha) \cdot h^{(2,0)}(\alpha+\theta,0),
    \end{align}
    which for $\alpha=0$ simplifies to
    \begin{align} \label{equation:ODE_preliminary_form}
        h^{(0,2)}(\theta,0) \cdot h(\theta,0)^2+2h^{(0,1)}(\theta,0) \cdot h^{(1,0)}(\theta,0)=-h(\theta,0)+h^{(2,0)}(\theta,0).
    \end{align}
    In the following, we write $h(\theta) \coloneqq h(\theta,0)$ for simplicity. Because every $u(\theta)$ for $\theta \in (-\varepsilon,\varepsilon)$ is $\lambda$-flat, we have
    \begin{equation}
        \frac{\lambda}{1-\lambda}g''_{\tilde K,u(\theta)}(0)=f''_{\tilde K,u(\theta)}(0)
    \end{equation}
    for all $\theta \in (-\varepsilon,\varepsilon)$. Using that
    $|f'_{\tilde K,u(\theta)}(\pi_{u(\theta)^\bot}(0))|=|\!\tan(\theta)|$, Lemma~\ref{lemma:curvature_formula} yields the relation
    \begin{equation} \label{equation:relation_for_second_derivative}
        h^{(0,2)}(\theta,0) = -\frac{1-\lambda}\lambda \cdot \kappa(\tilde K,0)\cdot(1+\tan(\theta)^2)^{3/2} = -\frac{1-\lambda}\lambda \cdot \frac{\kappa(\tilde K,0)}{\cos(\theta)^3}.
    \end{equation}
    If we plug \eqref{eq_first_derivative_of_g_alpha=0} and \eqref{equation:relation_for_second_derivative} into \eqref{equation:ODE_preliminary_form}, we obtain the second-order ordinary differential equation
    \begin{equation} \label{eq:second_order_ODE}
        -\frac{1-\lambda}\lambda \cdot \frac{\kappa(\tilde K,0)}{\cos(\theta)^3}\cdot h(\theta)^2+2 \frac{(h'(\theta))^2}{h(\theta)}+ h(\theta) = h''(\theta).
    \end{equation}
    Because $h(\theta)>0$ for $\theta \in (-\varepsilon,\varepsilon)$ and $\varepsilon<\frac\pi2$, the left-hand side is locally Lipschitz continuous and the Picard-Lindelöf theorem implies that \eqref{eq:second_order_ODE} uniquely determines $h$, given a pair of initial values $h(0)$ and $h'(0)$. Recalling property \ref{item:property1}, we know from \eqref{eq:property1_width} that
    \begin{equation}
        h(0)=-g_{\tilde K,e_2}(0)=w(\tilde K,e_2).
    \end{equation}
    Combining \eqref{eq_first_derivative_of_g_alpha=0} with \eqref{eq:property1_derivatives}, we get
    \begin{equation}
        h'(0)= h^{(0,1)}(0,0) \cdot h(0)=-g'_{\tilde K, e_2}(0)\cdot h(0)=0.
    \end{equation}
    We also recall from \eqref{eq:formula_for_kappa_at_origin} that
    \begin{equation}
        \kappa(\tilde K,0)= \frac{2}{w(\tilde K, e_2)}\cdot \frac\lambda{1-\lambda}=\frac{2}{h(0)}\cdot \frac\lambda{1-\lambda}.
    \end{equation}
    Plugging this into the differential equation \eqref{eq:second_order_ODE}, we get
    \begin{equation}
         h''(\theta)=-\frac{2}{h(0)}\cdot\frac{h(\theta)^2}{\cos(\theta)^3}+2 \frac{(h'(\theta))^2}{h(\theta)}+ h(\theta).
    \end{equation}
    Given the initial values $h(0)=w(\tilde K,e_2)$ and $h'(0)=0$, the unique solution is given by
    \begin{equation}
        h(\theta)=h(0) \cdot \cos(\theta)= w(\tilde K,e_2) \cdot\cos(\theta),
    \end{equation}
    because we have
    \begin{equation}
         -h(0)\cdot\cos(\theta)=-\frac{2}{h(0)}\cdot h(0)^2\cdot\frac{\cos(\theta)^2}{\cos(\theta)^3}+2\cdot h(0)\cdot \frac{\sin(\theta)^2}{\cos(\theta)}+ h(0)\cdot\cos(\theta)
    \end{equation}
    for $\theta \in (-\varepsilon,\varepsilon)$.
    
    We recall that $z(\theta)=h(\theta) u(\theta)$. Because $h(\theta)=\omega\cos(\theta)=w(\tilde K,e_2)  \cos(\theta)$, the set
    \begin{equation}
        A\coloneqq \{z(\theta)\mid \theta \in(-\varepsilon,\varepsilon)\}=\bd \tilde K \cap \left(\pos \{u(\theta)\mid \theta \in (-\varepsilon,\varepsilon)\}\setminus\{0\}\right)
    \end{equation}
    is an arc of a circle. This circle contains $0$ and $h(0) e_2$ as antipodal points, it is centered at $\frac{w(\tilde K,e_2)}{2} e_2$ and has radius $\frac{w(\tilde K,e_2)}{2}$. 
    
    We assume towards a contradiction that $\lambda < \frac12$. Because the $\lambda$-section $H^\lambda_{e_2}(\tilde K)$ is parallel to $e_2^\bot$, the set
    \begin{equation}
        A' \coloneqq [\bd \tilde K \cap \pi^{-1}_{e_2^\bot}(\pi_{e_2^\bot}(A))] \setminus A
    \end{equation}
    is an arc of an ellipse, which is symmetric around $\linh\{e_2\}$, and whose principal axes have lengths
    \begin{equation}
        w(\tilde K,e_2) \quad \text{and} \quad \frac\lambda{1-\lambda}\cdot w(\tilde K,e_2)<w(\tilde K,e_2).
    \end{equation}
    Since the curvature of an ellipse is minimized precisely at its co-vertices, the curvature $\kappa(\tilde K,\cdot)$ is uniquely minimized on $A'$ at the origin. By Lemma~\ref{lemma:curvature_formula} and \eqref{eq:property1_derivatives}, this minimum satisfies
    \begin{equation}
        \kappa(\tilde K,0)=f_{\tilde K,e_2}''(0)=\frac\lambda{1-\lambda}\cdot g_{\tilde K,e_2}''(0)=\frac\lambda{1-\lambda}\cdot \kappa_A,
    \end{equation}
    where $\kappa_A$ denotes the constant curvature of the circular arc $A$. Now let $x\in A \setminus\{\omega e_2\}$ be sufficiently close to $\omega e_2$ such that $n(\tilde K,x) \in \{u(\theta)\mid \theta \in (-\varepsilon,\varepsilon)\}$. Because $A$ is a circular arc, the line
    \begin{equation}
        x+ \linh \{n(\tilde K,x)\}
    \end{equation}
    intersects the elliptic arc $A'$ at some point $x'\neq 0$. Writing $\tilde u \coloneqq n(\tilde K,x)$, we have $g'_{\tilde K,\tilde u}(\pi_{\tilde u^\bot}(x))=0$. Using Lemma~\ref{lemma:curvature_formula} and the fact that $\tilde u$ is $\lambda$-flat, we obtain
    \begin{align}
        \kappa(\tilde K, x')=\frac{f_{\tilde K,\tilde u}''(\pi_{\tilde u^\bot}(x))}{ (1+ (f_{\tilde K,\tilde u}'(\pi_{\tilde u^\bot}(x)))^2)^{3/2}}
         \leq f_{\tilde K,\tilde u}''(\pi_{\tilde u^\bot}(x)) = \frac\lambda{1-\lambda} \cdot g_{\tilde K,\tilde u}''(\pi_{\tilde u^\bot}(x))&= \frac\lambda{1-\lambda} \cdot \kappa_A\\&=\kappa(\tilde K,0),%
    \end{align}
    in contradiction to the fact that $\kappa(\tilde K,\cdot)$ is uniquely minimized on $A'$ at $0$. It follows that $\lambda=\frac12$.
\end{proof}

It is now easy to prove Theorem~\ref{theorem:characterization_of_ellipses_and_triangles}.

\begin{proof}[Proof of Theorem~\ref{theorem:characterization_of_ellipses_and_triangles}]
    Let $K \in \Ktt$ be $\lambda$-invariant for some $\lambda\in[0,\frac12]$. The claim that $\lambda \in \{0,\frac12\}$ follows from Proposition~\ref{proposition:if_there_exists_an_arc_then_lambda_is_1/2}. That $K$ has to be a triangle in the first case and an ellipse in the second case follows from Lemma~\ref{lemma:sigma_invariance_and_sigma_universality}.
\end{proof}

\section{Proof of Theorem~\ref{theorem:solution_to_planar_question}}

In order to resolve the planar case of Question~\ref{question:precise_question}, we need one final ingredient.

\begin{proposition} \label{proposition:if_there_exists_S0_invariant_body_containment_of_SL2}
    Let $G\leq\GL(2,\R)$ be closed. If there exists a body $L^*\in\Ktt$ that is $S^0$-invariant modulo $G$, then $G$ contains $\SL(2,\R).$
\end{proposition}

\begin{proof}
    Writing $\SL^\pm(2,\R)\coloneqq \{A \in \GL(2,\R) \mid |\!\det A|=1\}$, we set $H \coloneqq G \cap\SL^\pm(2,\R)$. The claim follows if we show that $\SL(2,\R)\leq H$. By \cite[Rem.~2.2]{ccg01}, any triangle with non-obtuse angles is $S^0$-universal. Because shaking preserves volume, the body $L^*$ is $S^0$-invariant modulo $H$. Now Lemma~\ref{lemma:sigma_invariance_and_sigma_universality}  implies that every pair of non-obtuse triangles $T_1,T_2$ with $\vol T_1=\vol T_2$ satisfies $T_1,T_2\in [rL^*]_H$ for a suitably chosen $r>0$ and hence $T_1\in[T_2]_H$. For $i \in \N$, let $U_i\coloneqq \{M \in \SL^\pm(2,\R) \mid \norm{M-I_2}<\frac1i\}$. We assume towards a contradiction that there exists a sequence $(M_i)_{i\in\N}$ with $M_i \in U_i$ but $M_i \notin H$ for all $i\in \N$. Let $T \in \Ktt$ be a regular triangle and let $T_i=M_iT$ for $i \in \N$. Then there exists an index $i^*$ such that all $T_i$ with $i\geq i^*$ are non-obtuse. For all $i\geq i^*$, we have $T_i \in [T]_H$ and hence $M_i P_i \in H$, where $P_i \in \SL^\pm(2,\R)$ corresponds to a permutation of the vertices of $T$. But there are only six possibilities for $P_i$, which implies that there exists a subsequence $(M_{k(i)})_{i\in\N}$ such that $P_{k(i)}=P_{k(j)}\coloneqq P^*$ for all $i,j \in \N$. Since $(M_{k(i)})\goestoinfty{i}I_2$, we have $I_2 P^* \in H$ and hence $(P^*)^{-1}\in H$. Recalling $M_{k(i)}P^*\in H$, this implies $M_{k(i)}\in H$ for all $i \in \N$, which is a contradiction. It follows that there exists an index $i \in \N$ with $U_i \subset H$. But this implies that $H= \bigcup_{A \in H}U_iA$ is a union of open sets in $\SL^{\pm}(2,\R)$ and hence open in $\SL^{\pm}(2,\R)$. Since $H$ is also closed, $H$ contains the connected component of $I_2$ in $\SL^\pm(2,\R)$, which is $\SL(2,\R)$.
\end{proof}

The proof of Theorem~\ref{theorem:solution_to_planar_question} is now merely a matter of combining previous results.
\begin{proof}[Proof of Theorem~\ref{theorem:solution_to_planar_question}]
    We consider the random $\lambda$-Steiner process starting at $K$.
    
    Let $\lambda \in (0,\frac12)$ and let $G\leq\GL(2,\R)$ be closed. By Theorem~\ref{theorem:characterization_of_ellipses_and_triangles}, there exists no $\lambda$-invariant body $L^* \in \Ktt$. Since $G\leq \GL(2,\R)$, it follows that there exists also no $L^* \in \Ktt$ that is $S^\lambda$-invariant modulo $G$. Hence, $K$ is $S^\lambda$-errant modulo $G$ and, using Proposition~\ref{proposition:planar_lambda_steiner_symmetrization_is_compactly_recurrent}, Theorem~\ref{theorem:sufficient_condition_errant} implies that the random $\lambda$-Steiner process diverges almost surely modulo $G$.

    Let $\lambda =\frac12$ and let $G\leq\GL(2,\R)$ be closed. By Corollary~\ref{corollary:convergence_to_ball_for_lambda=1/2}, the random Steiner process starting at $K$ converges to any Euclidean disk $B$ with $\vol K=\vol B$ almost surely. Since $G$ contains $\{\Id\}$, the process almost surely converges modulo $G$.

    Finally, for $\lambda = 0$, the claim about almost sure convergence follows from Corollary~\ref{corollary:random_shaking_process_converges_to_any_triangle}, whereas the claim about almost sure divergence follows from Theorem~\ref{theorem:sufficient_condition_errant}, using  Proposition~\ref{proposition:planar_lambda_steiner_symmetrization_is_compactly_recurrent} and Proposition~\ref{proposition:if_there_exists_S0_invariant_body_containment_of_SL2}.
\end{proof}

\section{Concluding remarks}

Since Question~\ref{question:precise_question} makes sense for arbitrary $n\geq 1$, whereas Theorem~\ref{theorem:solution_to_planar_question} only addresses the case $n=2$, the most obvious question that has been left open in the preceding section is how our results might be extended to the case $n \geq 3$. Although the arguments in Section~\ref{section:compact_recurrence_of_planar_lambda_steiner_processes} are very specific to the planar case, it seems likely that the main result of this section holds in higher dimensions as well.

\begin{conjecture} \label{conjecture:lambda_steiner_symmetrization_is_compactly_recurrent}
    For every $\lambda \in [0,\frac12]$ and $n \geq 1$, $\lambda$-Steiner symmetrization is compactly recurrent in $\R^n$.
\end{conjecture}

If we assume the truth of Conjecture~\ref{conjecture:lambda_steiner_symmetrization_is_compactly_recurrent}, then as in the planar case, the criteria from Section~\ref{section:random_convergence_of_generalized_steiner_processes} resolve Question~\ref{question:precise_question} for all triples $(K,\lambda,G)$ where $K$ is $S^\lambda$-errant or $S^\lambda$-tethered modulo $G$. We know that every convex body in arbitrary dimension is $S^{\frac12}$-tethered modulo $\{\Id\}$. On the other hand, since the $n$-dimensional simplex fails to be $0$-invariant for $n \geq 3$, Lemma~\ref{lemma:sigma_invariance_and_sigma_universality} implies that there exists no $0$-invariant $n$-dimensional convex body for $n \geq 3$, and hence every $K \in \Knn$ is $S^0$-errant modulo $G$ for $n \geq 3$ and arbitrary closed subgroups $G$ of $\GL(n,\R)$. It remains unresolved what happens for $\lambda \in (0,\frac12)$; here we conjecture that no $\lambda$-invariant body exists or, equivalently, that the following characterization of ellipsoids holds.

\begin{conjecture} \label{conjecture:characterization_of_ellipsoids}
    Let $n \geq 3$, $K \in \Knn$ and $\lambda \in [0,\frac12]$ such that $K$ is $\lambda$-invariant. Then $\lambda=\frac12$ and $K$ is an ellipsoid.
\end{conjecture}

The main reason to believe  Conjecture~\ref{conjecture:characterization_of_ellipsoids} is the heuristic idea that $\lambda$-invariance should become a more restrictive property as $n$ increases (as witnessed by the simplex). Although it seems plausible that our arguments from Section~\ref{section:a_characterization_of_ellipses_and_triangles} might be extended to the case $n \geq 3$, additional technical difficulties are to be expected.
In this context, it makes sense to consider the following related question, which we have not even addressed in the planar setting.

\begin{question} \label{question:S_lambda_universal_body}
    Let $n \geq 2$ and $\lambda \in (0,\frac12)$. Does there exist an $S^\lambda$-universal body $K \in \Knn$?
\end{question}

If we could construct such a body for given $n$ and $\lambda$, then by Lemma~\ref{lemma:sigma_invariance_and_sigma_universality} it would be the only candidate for a $\lambda$-invariant body. In particular, such a construction in the planar case could render the arguments from Section~\ref{section:a_characterization_of_ellipses_and_triangles} obsolete. More generally, one might ask about the behavior of $\lambda$-Steiner symmetrization using finite sets of directions; for $\lambda=\frac12$, this topic was studied by Klain \cite{k12}.
Independently of the answer to Question~\ref{question:S_lambda_universal_body}, Conjecture~\ref{conjecture:lambda_steiner_symmetrization_is_compactly_recurrent} and Conjecture~\ref{conjecture:characterization_of_ellipsoids} imply the following conjectural answer to Question~\ref{question:precise_question}.

\begin{conjecture} \label{conjecture:conjectural_solution_to_precise_question}
    Let $n\geq 3$, $K \in \Knn$, $\lambda \in [0,\frac12]$ and let $G\leq \GL(n,\R)$ be closed. The random $\lambda$-Steiner process starting at $K$ converges modulo $G$ almost surely if and only if $\lambda=\frac12$.
    If $\lambda=\frac12$, the process converges modulo $G$ to (the equivalence class of) the Euclidean ball of volume $\vol K$. For all other values of $\lambda$, the process diverges modulo $G$ almost surely.
\end{conjecture}

Conjecture~\ref{conjecture:conjectural_solution_to_precise_question} asserts that in dimension $n \geq 3$, the convergence behavior of Steiner symmetrization is exceptional in the class of $\lambda$-Steiner symmetrization, whereas the behavior of shaking is generic, in contrast to the planar case. More broadly, the exact correspondence between Steiner symmetrization and shaking seems to be a peculiar feature of $\R^2$, and a natural counterpart to Steiner symmetrization seems to be unavailable in higher dimensions. This asymmetry makes the task of establishing reverse counterparts to isoperimetric-type inequalities very challenging. On the other hand, the exceptional nature of Steiner symmetrization might lead us to appreciate its usefulness in a new light. %

\subsection*{Acknowledgments} The research leading to these results is part of a project that has received funding from the European Research Council (ERC) under the European Union’s Horizon 2020 research and innovation programme (grant agreement No 101001677). The first-named author thanks the University of Szeged, where part of the research was done, for its hospitality.

This research was supported by NKFIH projects no.~150151 and no.~150613. Projects no.~150151 and no.~150613 have been implemented with the support provided by the Ministry of Culture and Innovation of Hungary from the National Research, Development and Innovation Fund, financed under the ADVANCED\_24 funding scheme.

The second author received funding from project no.~2024-1.2.8-T\'ET-IPARI-CN-2025-00011,
with the support provided by the Ministry of Innovation and Technology
of Hungary from the National Research, Development and Innovation Fund.

\subsection*{Statement on the use of artificial intelligence} The ideas presented in the article were developed by the authors without assistance from artificial intelligence. ChatGPT was used to search for references, to avoid terminological collisions with existing literature, and to correct errors in the manuscript.

    \printbibliography
\end{document}